\documentclass{amsart}
\usepackage[english]{babel}
 \usepackage[utf8]{inputenc}
 \usepackage{hyperref}
 \allowdisplaybreaks  

 \usepackage{amssymb, mathdots}
\usepackage[all]{xy}
\usepackage{enumerate}

\newcommand{\C}{{\mathbb C}}
\newcommand{\Gr}{\operatorname{Gr}}

\newcommand{\Z}{{\mathbb Z}}

\newcommand\End{\operatorname{End}}
\newcommand\Gl{\operatorname{Gl}}
\newcommand\Hom{\operatorname{Hom}}
\newcommand\homlie{\operatorname{Hom}_{\text{Lie-alg}}}
\newcommand\Tr{\operatorname{Tr}}

\newcommand\Vkdv{V_{\operatorname{odd}}}
\newcommand\coker{\operatorname{coker}}
\newcommand\ad{\operatorname{ad}}
\newcommand\uh[1][{}]{{{\mathcal U}({\mathbb H}#1})}
\renewcommand{\sl}{\mathfrak{sl}}

\newtheorem*{thm*}{Theorem}
\newtheorem{thm}{Theorem}[section]
\newtheorem{prop}[thm]{Proposition}
\newtheorem{defn}[thm]{Definition}
\newtheorem{cor}[thm]{Corollary}
\newtheorem{lem}[thm]{Lemma}
\newtheorem{rem}[thm]{Remark}
\newtheorem{exam}[thm]{Example}

\begin{document}

\title{Virasoro and KdV}


%
\author{Francisco J. Plaza Mart\'{\i}n}
\email{fplaza@usal.es}
\address{Departamento de Matem\'aticas and IUFFyM, Universidad de
Salamanca,  Plaza de la Merced 1.  37008 Salamanca. Spain. }
\author{Carlos Tejero Prieto}
\email{carlost@usal.es}
\address{Departamento de Matem\'aticas and IUFFyM, Universidad de
Salamanca,  Plaza de la Merced 1.  37008 Salamanca. Spain. }
%
\thanks{
\\
\indent This work is supported by the research grants Mineco-MTM2012--32342, MTM2013--45935--P and Fundaci\'on Sol\'orzano FS/9-2014}

\begin{abstract}
We investigate the structure of representations of the (positive half of the) Virasoro algebra 
and  situations in which they decompose as a tensor product of Lie algebra representations. As an illustration, we apply these  results to the differential operators defined by the Virasoro conjecture and obtain some factorization properties of the solutions as well as a link to  the multicomponent KP hierarchy.
\end{abstract}

\maketitle

\section{Introduction}

The breakthrough discovery of Witten-Kontsevich (\cite{Witten,Kon}) established an intimate link between mathematical physics and enumerative geometry. From a general perspective, one aims at studying the Gromov-Witten invariants of a smooth projective variety $X$  in terms of suitable integrable hierarchies. From this point of view, the Witten-Kontsevich case corresponds to the situation when $X$ is a point and it was shown that the exponential of the generating function of intersection numbers on the moduli space of curves was a common solution of the Virasoro constraints and of the KdV hierarchy. Therefore, following the generalization for the case of the projective space proposed in \cite{EHX}, on the one hand one wonders if the generating function fulfills a generalization of the Virasoro constraints. On the other hand, one also wants to know if the generating function is given by (the logarithm of) a particular tau-function of an integrable hierarchy. Nowadays, the case of varieties with semisimple quantum cohomology is well understood and the answer to both questions is affirmative (\cite{Teleman} and \cite{DubrovinZ2}; see also \cite{DubrovinZ, Getzler,Givental, LiuTian, OK, OK-Toda}). 


It is worth pointing out that, for each $X$, the explicit Virasoro operators as well as the relevant integrable hierarchy may vary; for instance, the $2$-Toda hierarchy appears when dealing with the equivariant GW invariants of ${\mathbb P}^1$ (\cite{OK-Toda}). Nevertheless, one recognizes some common features that arise among these results. Let us mention some of them. In \cite{GiventalAn}, Givental studied a case in which the total descendent potential is a $\tau$-function for the $n$KdV-hierarchy by using $n-1$ copies of the KdV. Thus,  the total descendent potential of a semisimple Frobenius manifold was defined in \cite{Givental} as a product of $n$ copies of Witten-Kontsevich $\tau$-functions. Dealing with a case of orbifold quantum cohomology, it has been proved in \cite{JarvisKimura} that the Virasoro constraints decomposed as $n$ copies of (half of) the Virasoro algebra, that their solution was the product of Witten-Kontsevich $\tau$-functions, and that the relevant integrable hierarchy consisted of $n$ commuting copies of the KdV hierarchy. Finally, in \cite{DVV, KacSchwarz} it was shown that  the solution of the Virasoro constraints in the case of Witten-Kontsevich is unique  (up to a constant factor) and this uniqueness also holds in other setups (e.g. \cite{LiuUnique}).

This paper, making use of the representation theoretic properties of the Virasoro algebra, 
offers new insights and results on these properties and provides evidences that the above mentioned properties rely heavily on the structure of the Virasoro algebra and its representations. Our study of explicit expressions for Virasoro representations (see  \S2) is general enough to encode many of the known representations within the framework of Virasoro constraints. Further, it allows us to determine whether a representation is the tensor product of Lie algebra representations and if a solution factorizes as a product of solutions of those representations. An explicit realization of these ideas is carried out in \S3 for the case of  smooth projective varieties with trivial odd cohomology and vanishing first Chern class. Thus, we think that our approach may help in determining the explicit expression of the Virasoro operators as well as the corresponding integrable hierarchies for other types of varieties $X$ (see \S\ref{subsec:final}).  Now, let us be more precise and explain the contents of the paper.

We begin by fixing a pair $(A,(\, ,\,))$ consisting of a finite dimensional vector space and a non-degenerate bilinear form. Associated to this data, we consider a Heisenberg algebra ${\mathbb H}(A)$ and its universal enveloping algebra $\uh[(A)]$. Let us denote by ${\mathcal W}_>$ the positive half of the Virasoro algebra and recall that it contains $\sl(2)$ canonically. Section 2 is entirely devoted to the study of Lie algebra maps ${\mathcal W}_>\to \uh[(A)]$. To begin with, we show that, under some homogeneity condition, there is a canonical bijection between $\Hom_{\text{Lie-alg}}( {\mathcal W}_>, \uh[(A)])$ and $\Hom_{\text{Lie-alg}}( \sl(2), \uh[(A)]$ (Theorem~\ref{thm:W^+-sl2}). This is highly non-trivial since, in general, the problem of extending a map defined on $\sl(2)$ to ${\mathcal W}_>$ involves  infinitely many conditions (see~\cite{PT}). Accordingly, it is natural to expect that many properties of a map ${\mathcal W}_> \to \uh[(A)]$ can be stated in terms of its restriction to $\sl(2)$. Actually, we prove that such a map decomposes as tensor product of Lie algebra representations if and only if its restriction does. Moreover, this factorization is possible only if $A$ decomposes as the orthogonal sum of  two subspaces (Theorem~\ref{thm:rhodecomp}). We finish this section by showing that the fact that ${\mathcal W}_>$ admits no non-trivial finite dimensional representations has important consequences for the structure of the solutions of the equations $(\rho_1\otimes 1 + 1\otimes \rho_2)(L)(\sum_i f_i\otimes g_i)=0$, where $L\in {\mathcal W}_>$ (see Theorem~\ref{thm:solutionrhoi}). That is, decompositions of the representation and of the solutions depend strongly on the structure of ${\mathcal W}_>$ and of $(A,(\, ,\,))$. 

Although the previous results are interesting on their own, \S3 explores their application to concrete situations; for instance, relations with integrable hierarchies (e.g. multicomponent KdV). The case we have chosen to illustrate this issue is that of the differential operators appearing in the Virasoro conjecture when $X$ has trivial odd cohomology (for instance, whenever $X$ has semisimple even quantum cohomology) and its first Chern class vanishes. Then, Theorem~\ref{thm:Lbar=Lhat} shows explicitly how to obtain these operators as the images of the generators $L_k\in{\mathcal W}_>$ by:
	{\small $$
	\hat\rho\,:\, {\mathcal W}_>\, \overset{\rho}\longrightarrow\,  \uh[(A)] \, \overset{\widehat{\,}}\longrightarrow\, \End\big(\C[[\{t_{i,\alpha}\vert 1\leq \alpha\leq \dim(A) , i=1,3,\ldots\}]]\big)
	$$}for $A=H^*(X,\C)$ endowed with the Poincar\'e pairing. Then, our results of \S2 imply that $\hat\rho$ decomposes as the tensor product of Lie algebra representations associated to data $(\C,(\, , \,))$; i.e. the $1$-dimensional case. The detailed study of the $1$-dimensional case carried out in \S\ref{subsec:solutions1dim} shows that, up to re-scaling the variables, the corresponding operators \emph{always }come from a representation:
	$$
	\sigma: {\mathcal W}_> \,\longrightarrow \, \operatorname{Diff}^1(\C((z)))
	$$
which means that we can profit from \cite{KacSchwarz, Plaza-AlgSol} to build the unique solution in terms of a $\tau$-function of the KdV hierarchy. Putting everything together, we have the main results of this section. First, in the case of $\dim A=1$:

\begin{thm*}[see Theorem~\ref{thm:solutionWscaleKdV}]
Let $\rho \in \homlie({\mathcal W}_>,\End(\C[[t_1,t_3,\ldots]]))$ be such that $\rho(L_k)$  is of type $k$ for $k\geq -1$ and that all coefficients of $\rho(L_{-1})$ are non zero. 

Then, there exists a unique $\tau(t)\in\C[[t_1,t_3,\ldots]]$, with $\tau(0)=1$, such that:
	$$
	\rho(L_k)(\tau(t))\,=\,0\qquad k\geq -1
	$$
Further, the solution $\tau(t)$ is a $\tau$-function of the scaled KdV hierarchy.
\end{thm*}\noindent and, for $\dim A=N \geq 2$:

\begin{thm*}[see Theorem~\ref{thm:solutionproductKdV}]
Let $\rho:{\mathcal W}_>\to \uh[(A)]$ be as in \S\ref{subsec:baby}.

There exist $S\in \Gl(A)$ and functions $\tau_\alpha(t_{1,\alpha},t_{3,\alpha},\ldots)\in \C[[t_{1,\alpha},t_{3,\alpha},\ldots]]$ such that:
	$$
	\hat\rho(L_k)\big(S( \prod_\alpha \tau_\alpha( t_\alpha))\big)\,=\, 0
	$$

Further, $\tau_\alpha(t_{1,\alpha},t_{3,\alpha},\ldots)$ are $\tau$-functions of the scaled KdV hierarchy.\end{thm*}

We hope that our methods shed some light on the explicit expressions of the Virasoro operators and of the relevant integrable hierarchies that appear when studying the Virasoro conjecture.  We also think that the techniques presented here can be applied to many instances of representations of ${\mathcal W}_>$ which appear in a variety or problems such as recursion relations, Hurwitz numbers, knot theory, etc.~. We sketch some ideas in \S\ref{subsec:final} although all of them deserve further research.  

\emph{Acknowledgements.} We thank A. Givental and G. Borot for explaining us some facts on their papers. The first author  wish to express his gratitude to the Max Planck Institute f\"ur Mathematik (Bonn) for the invitation in fall 2014. 

\section{Lie algebras}\label{sec:ReprW>}

Let ${\mathcal W}$ be the Witt algebra; that is, the $\C$-vector space
with basis $\{L_k\}_{k\in\Z}$ endowed with the Lie bracket
$[L_i,L_j]=(i-j) L_{i+j}$, and let ${\mathcal W}_>$ be the subalgebra
generated by $\{L_k\}_{k\geq -1}$. It contains a copy of $\sl(2)$ via $\sl(2)=<\{L_{-1}, L_0, L_1\}> \subset {\mathcal W}_>$. Recall that ${\mathcal W}_>$ is also called the \emph{positive half of the centereless Virasoro algebra}. 

In this section, we study certain maps from  $\sl(2)$ and their extensions to  ${\mathcal W}_>$. These results will eventually allow us to relate the representation theories of ${\mathcal W}_>$ and $\sl(2)$. A further consequence is that, in order to construct the operators $L_0, L_1, L_2, \ldots$ one only has to start with $L_{-1}$ and follow some simple procedures and choices. 

It is worth mentioning that a study of the representation theory of ${\mathcal W}_>$ in terms of the representation theory of its subalgebra $\sl(2) \subset {\mathcal W}_>$ has been carried out in \cite{PT} in full generality. 
	
\subsection{Preliminaries}\label{subsec:prelim-beta-Diff-to-C[[t]]}

Let us be more precise. Let $(A, (\, ,\,))$ be given, where $A$ is a finite dimensional $\C$-vector space and $(\,,\,)$ is a non-degenerated bilinear pairing. For a basis $\{a_{\alpha}\vert \alpha=1,\ldots, n\}$ of $A$, let $\eta=(\eta_{\alpha\beta})$ denote the matrix associated to the given bilinear product; that is, $\eta_{\alpha\beta}:=(a_{\alpha}, a_{\beta})$. The inverse will be denoted with superindexes; i.e. $\eta^{\alpha\beta}:=(\eta^{-1})_{\alpha\beta}$. 

Let us consider unknowns $\{p_i,q_i\vert i\geq 1\}$ and introduce $p_{i,\alpha}:=p_i\otimes a_{\alpha}$ and $q_{i,\alpha}:=q_i\otimes a_{\alpha}$. Let ${\mathbb H}(A)$ be the Heisenberg algebra generated by $\{1, p_{i,\alpha},q_{i,\alpha}\vert i\geq 1, \alpha=1,\ldots, n\}$, whose elements will be called \emph{operators},  endowed with the Lie bracket: 
    \begin{equation}\label{eq:pqcommutation}
    \begin{gathered}
    \, [p_{i,\alpha}, q_{j,\beta}] \,=\, \delta_{i,j} i \eta^{\alpha\beta}\cdot 1
    \\
   \,  [p_{i,\alpha}, p_{j,\beta}] \,=\,[q_{i,\alpha}, q_{j,\beta}] \,=\, 0
    \\
   \,  [p_{i,\alpha}, 1] \,=\,[q_{i,\alpha}, 1] \,=\, 0
    \end{gathered}
    \end{equation}
We define their degree by $\deg(q_{i,\alpha})=i$, $\deg(p_{i,\alpha})=-i$ and $\deg(1)=0$. 

Although the definition of the Heisenberg algebra depends on the pair $(A, (\, ,\,))$, it will be simply denoted by ${\mathbb H}$ if no confusion arises.

For ${\mathbb H}$ as above, let us define $\uh$ the universal enveloping algebra of ${\mathbb H}$, which is the quotient of the tensor algebra of  ${\mathbb H}$ by the two-sided ideal generated by the relations $u\otimes v- v\otimes u-[u,v]$. 

Motivated by the explicit forms of the Virasoro operators considered in the literature (\cite{DVV, DubrovinZ,EO, Getzler, Givental,KSU, KazarianZograf}), we introduce the following notion. The ultimate meaning of this notion is unveiled in Lemma~\ref{lem:Diff1scaleW}.
 
%

\begin{defn}\label{defn:type}
An operator $T\in \uh$ is \emph{of type $i\geq -1$} if it is a linear combination of $p_{2i+3,\alpha}$ and double products of degree $-2i$; i.e. $p_{j,\alpha}p_{2i-j,\beta}$, $q_{j,\alpha}p_{2i+j,\beta}$ and $q_{j,\alpha}q_{-j-2i,\beta}$. If $i=0$ we also allow a constant term.

The subset consisting of operators of type $i\geq -1$ will be denoted by $\uh[(A)]_i$ (or, simply, $\uh_i$). 
\end{defn}

This section deals with the study of homomorphisms of Lie algebras:
	$$
	\rho:{\mathcal W}_> \, \longrightarrow\, \uh \qquad \text{s.t. } \rho(L_i)\in \uh_i
	$$
	
Let us illustrate the previous definition. From now on, according to Einstein convention, summation over repeated indices will be understood. For instance, an operator of type $-1$ is of the form:
    \begin{equation}\label{eq:type-1}
    b_{-1}^{0,1} p_1 +  q_1 a_{-1}^{1,1}  q_1^T  + q_{i+2} b_{-1}^{i+2,i}  p_i \, \in\, \uh_{-1} 
    \end{equation}
(the sum  runs over the set of odd positive integers $i$),  $p_i$ is the column vector $(p_{i,1}\ldots,p_{i,n})^T$, $q_i$ is the row vector $(q_{i,1}\ldots,q_{i,n})$,  $b_{-1}^{0,1}$ is a row vector, $a_{-1}^{1,1}$  and $b_{-1}^{i+2,i}$ are $n\times n$ matrices. For brevity, we set $a:=a_{-1}^{1,1} $. 

Similarly, an type $0$ operator can be expressed as:
    \begin{equation}\label{eq:type0}
   b_{0}^{0,3} p_3 + b_0^{0,0}  + q_{i} b_{0}^{i,i}  p_i \, \in\, \uh_{0}
    \end{equation}
while an operator of type $i\geq 1$ is of the form:
	\begin{equation}\label{eq:type}
	b_i^{0,2i+3} p_{2i+3} + p_{j}^T c_i^{j,2i-j} p_{2i-j} +  q_j b_i^{j,2i+j} p_{2i+j}
	\,\in\,\uh_i \qquad i\geq 1
	\end{equation}
for a row vector $b_i^{0,2i+3}$ and $n\times n$-matrices $b_i^{j,2i+j}$ and $c_i^{j,2i-j}$, where  $c_i^{j,2i-j}=(c_i^{2i-j,j})^T$ and the sum runs over $j$ odd.

It is convenient to offer an interpretation of these matrices. Recall that $q_i$ is the row vector $(q_{i,1},\ldots, q_{i,n})$, which can be thought as an ${\mathbb H}$-valued vector of $A$. A similar argument holds for the column vector $p_i$. Thus, under a basis change in $A$, the matrix $b$ in $q_i\cdot b\cdot p_i$ behaves as a bilinear form on $A$. The same fact applies to all $a$, $b$ and $c$ matrices.  Similarly, column vectors $b_i^{0,2i+3}$ are understood as vectors on $A$ while row vectors are like linear forms. 

It is worth noticing how these operators behave w.r.t. the Lie bracket. Indeed, the computations given in subsection~\S\ref{subsect:CR} and the linearity of the bracket show that it is compatible with the degree:
	\begin{equation}\label{eq:LieT}
	[\,\, , \,\, ] \,:\, \uh_i \times \uh_j \, \longrightarrow\, \uh_{i+j}
	\end{equation}
where $i,j,i+j\geq -1$. In particular, it follows that $\oplus_{i\geq -1} \uh_i$ is a partial Lie algebra.

\subsection{Maps from $\sl(2)$ to Heisenberg}\label{subsect:sl2}

Let $\sl(2)$ be the Lie algebra of $\operatorname{Sl}(2,\C)$. We fix a  basis $\{e,f,h\}$  of $\sl(2)$  satisfying the relations:
    $$ 
    [e,f]=h\; , \qquad
    [h,e]=2e \; , \qquad
    [h,f]=-2f\; .
    $$ 
In particular, the previous choice yields a natural embedding:
    \begin{equation}\label{eq:emb-sl2-W+}
    \iota:\sl(2)\,\hookrightarrow\, {\mathcal W}_>
    \end{equation}
by mapping $f$ to $L_{-1}$, $h$ to $-2L_0$, and $e$ to $-L_1$.

\begin{lem}\label{lem:FyieldsH}
Let $F\in  \uh_{-1}$ be as in \eqref{eq:type-1}. 
Assume that 
$b_{-1}^{i+2,i}$ is invertible for all $i$. It holds that:
    {\small $$
    \left\{
    \begin{gathered}
    H\in\uh_0 \text{ s.t.}
    \\
    [H,F]\,=\, - 2F
    \end{gathered}
    \right\}
    \,\simeq\,
    \left\{
    \begin{gathered}
    (b,B)\in \C\times \operatorname{Mat}_{n\times n}(\C)
    \text{ s.t.}
    \\ (B \eta^{-1} + \operatorname{Id}) (a +a^T) +  (a +a^T)(B \eta^{-1} +\operatorname{Id})^T = 0
    \end{gathered}
    \right\}
    $$}
\end{lem}

\begin{proof}
Our task consists of computing the  bracket $[H,F]$ explicitly. Recall that, for simplicity, we have set $a=a_{-1}^{1,1}$. Since $H\in\uh_0$, it must be of the form $H:=  b_0^{0,3} p_3 +  b_{0}^{0,0}   +  q_{i}  b_{0}^{i,i} p_i $ where $b_0^{0,3}$ is a row vector, $b_{0}^{0,0}$ is an homothety, and $b_{0}^{i,i}$ are $n\times n$ matrices.

Having in mind the commutation relations of \S\ref{subsect:CR}, the bracket $[H,F]$ is a linear combination of $p_{1}$, $q_{1\alpha}q_{1\beta}$ and $q_{i+2,\alpha}p_{i,\beta}$. Therefore, the expression $[H,F]=-2F$ is equivalent to the following identities:
	$$
	    \begin{aligned}
    \big( 3 b_0^{0,3} \eta^{-1}  b_{-1}^{3,1} - b_{-1}^{0,1}\eta^{-1}  b_0^{1,1} \big) p_1 \, & =\, -2 b_{-1}^{0,1} p_1
    \\
     q_1 b_0^{1,1}\eta^{-1}  \big( a  + a^T   \big) q_1^T   \, & =\,  -2 q_1 a   q_1^T
    \\
    q_{i+2}\big( (i+2) b_0^{i+2,i+2}\eta^{-1}  b_{-1}^{i+2,i}  - i  b_{-1}^{i+2,i}\eta^{-1}  b_0^{i,i}\big) p_i
    \, & =\,  - 2 q_{i+2} b_{-1}^{i+2,i} p_i
    \qquad\forall i \geq 1
    \end{aligned}
	$$
Observe that $q_1 A q_1^T = q_1 B q_1^T$ if and only if $A+A^T=B+B^T$.  Hence, 
the above system is equivalent to the following equations:
        \begin{subequations}\label{eq:FyieldsH}
    \begin{align}
    \label{eq:FyieldsH:1}
    3 b_0^{0,3} \eta^{-1} b_{-1}^{3,1} - b_{-1}^{0,1}\eta^{-1} b_0^{1,1} \, & =\, -2 b_{-1}^{0,1}
    \\
     \label{eq:FyieldsH:2}
    b_0^{1,1} \eta^{-1} (a +a^T)+  (a +a^T)( b_0^{1,1}\eta^{-1} )^T  
     \, & =\,  - 2 (a +a^T)
         \\
    \label{eq:FyieldsH:3}
    (i+2) b_0^{i+2,i+2}\eta^{-1} b_{-1}^{i+2,i}  - i b_{-1}^{i+2,i} \eta^{-1}  b_0^{i,i}
    \, & =\,  - 2 b_{-1}^{i+2,i}
    \qquad\forall i\geq 1 
    \end{align}
        \end{subequations}

Note that, since $b_{-1}^{i+2,i}$ and $\eta$ are invertible,  given a pair $(b,B)$ as in the statement, this system has a unique solution for $b_0^{0,0}=b$ and $b_0^{1,1}=B$; namely,
	\begin{equation}\label{eq:coefH}
	\begin{aligned}
	b_0^{0,3}\, &=\, \frac13 b_{-1}^{0,1}(\eta^{-1} b_0^{1,1}-2) (\eta^{-1} b_{-1}^{3,1})^{-1} 
	\\
	b_0^{i+2,i+2} \,&=\, \frac1{i+2}  b_{-1}^{i+2,i} (i \eta^{-1} b_0^{i,i}-2 )(\eta^{-1} b_{-1}^{i+2,i})^{-1} \qquad\forall i\geq 1
	\end{aligned}
	\end{equation}
The converse is straightforward.
\end{proof}

\begin{exam} 
Set  $F= b_{-1}^{0,1}  p_1 +   q_1  a  q_1^T  +  \frac{i+2}2 q_{i+2} p_i$ and $b_0^{1,1}=-\frac12$, then  $H=-2 b_{-1}^{0,1} p_3 +   b_{0}^{0,0}   + i q_i p_i$.  Note that $i q_{i} p_i$ is the \emph{degree operator}.
\end{exam}

\begin{exam} 
Let us consider the case where the chosen basis in $A$ is orthonormal; i.e. $\eta$ is the identity matrix, and suppose that:
    $$
    F\,=\, b_{-1}^{0,1}  p_1 +   q_1  a  q_1^T  +  q_{i+2} p_i\,\in\,\uh_{-1}
    $$
Then, operators $H$ given by Lemma~(\ref{lem:FyieldsH}) acquire the form:
    $$ 
    H \,=\,    \frac13 b_{-1}^{0,1}( b_0^{1,1}-2)   p_3 
    +   b_{0}^{0,0}   +  \frac1{i} q_i  ( b_0^{1,1}-(i-1)) p_i,\in\,\uh_0
    $$ 
where $ b_{0}^{0,0} \in \C$ and $b_0^{1,1}$ verifies $(b_0^{1,1} +\operatorname{Id}) a  +  a (b_0^{1,1} +\operatorname{Id})^T =0$. 
\end{exam}

\begin{exam}\label{exam:dim1}
Finally, let $\dim A=1$ , $a, \eta\in\C^{\ast}$ and $F = b_{-1}^{0,1}  p_1 +   q_1  a  q_1^T  +  q_{i+2} \eta p_i$. Then, $ b_0^{i,i}=-\eta$ for all $i$ and $H=-  b_{-1}^{0,1}  p_3     +   b_{0}^{0,0}   -   q_i  \eta  p_i$. 
\end{exam}

\begin{lem}\label{lem:T'T'bracket}
Let $H$ be as in equation~\eqref{eq:type0} and $\uh'_i$ be the subspace:
	$$
	\uh'_i\,:=\, \{ T\in\uh_i  \text{ s.t. } [H,T]= 2i T \}
	$$
Then, it holds that $[ \uh'_i ,  \uh'_j ]\subseteq \uh'_{i+j} $.
\end{lem}

\begin{proof}
The claim follows easily from~\eqref{eq:LieT} and the Jacobi identity.
\end{proof}

\begin{thm}\label{thm:sl2-symmmatrices}
Let $F$ and $H$ be as in equations \eqref{eq:type-1} and \eqref{eq:type0} respectively. 

There is a surjective map:
    {\small
    $$
    \left\{
        \begin{gathered}
            c \in M_{n\times n}(\C)\text{ such that }
            \\
             b_0^{0,0} =   \operatorname{Tr}( c \eta^{-1}  (a+a^T)  (\eta^{-1})^T)
            \\
            \text{and equation~\eqref{eq:HF-2F:2} below}
        \end{gathered}
    \right\} 
    \, \longrightarrow\,
    \left\{
        \begin{gathered}
        \sigma \in \homlie(\sl(2),\uh) \\
         \text{such that } \sigma(f)=F, \\
         \sigma(h)=H  \text{ and }\sigma(e)\in\uh_1
        \end{gathered}
    \right\}
    $$}
Moreover, $c_1$ and $c_2$ have the same image iff $c_1+c_1^T= c_2+c_2^T$. Thus, the restriction of the above map to symmetric matrices yields a bijection. 
\end{thm}

\begin{proof}
Giving a map $\sigma$ as in the r.h.s. is equivalent to set an operator $E\in\uh_1$, such that $[E,F] = H $ and $ [H,F]= -2F$. Consider:
     \begin{equation}\label{eq:E=type1}
     E\,=\,  b_1^{0,5} p_5 + p_1^T c_1^{1,1}p_1  + q_i  b_1^{i,i+2}  p_{i+2}\,\in\, \uh_1
     \end{equation}
where, for simplicity, we will set $c= c_1^{1,1}$. The identity $[H,E]= 2E$, expressed in terms of the coefficients of the operators, is equivalent to the following equations (thanks to the computations of \S\ref{subsect:CR}):
    \begin{subequations}
    \label{eq:HF-2F}
    \begin{align}
     3 b_{0}^{0,3}\eta^{-1} b_1^{3,5} - 5 b_{1}^{0,5}\eta^{-1} b_0^{5,5} 
     \, & =\,  2 b_1^{0,5}\label{eq:HF-2F:1}
    \\
    - (c^T + c) \eta^{-1}b_0^{1,1} - (\eta^{-1}b_0^{1,1})^T  (c^T + c)  
    \, & =\, 2 (c^T + c) \label{eq:HF-2F:2}
    \\
    r b_0^{r,r}  \eta^{-1}  b_1^{r,r+2} - (r+2) b_1^{r,r+2} \eta^{-1}  b_0^{r+2,r+2}
    \,&=\, 2 b_1^{r,r+2} \label{eq:HF-2F:3}
    \end{align}
    \end{subequations}

Analogously, expanding the relation $[E,F]=H$ with the help of \S\ref{subsect:CR} yields the system:
    \begin{subequations}
    \label{eq:EFH}
    \begin{align}
    \label{eq:EFH:1}
    \operatorname{Tr}( c \eta^{-1} (a+a^T) (\eta^{-1})^T ) \, & = \,  b_0^{0,0}
    \\
    \label{eq:EFH:2}
    - b_{-1}^{0,1}\eta^{-1}  b_1^{1,3} + 5  b_1^{0,5} \eta^{-1} b_{-1}^{5,3} \, & = b_0^{0,3} 
    \\
    \label{eq:EFH:3}
    3 b_1^{1,3}\eta^{-1} b_{-1}^{3,1}   +  (a+a^T) (  \eta^{-1})^T  (c+c^T )  \, & = \,  b_0^{1,1}
    \\
    \label{eq:EFH:4}
    (r+2) b_1^{r,r+2} \eta^{-1} b_{-1}^{r+2,r} - (r-2) b_{-1}^{r,r-2} \eta^{-1}   b_1^{r-2,r} \, & = \, b_0^{r,r} \quad\forall r>2
    \end{align}
    \end{subequations}
Having in mind the properties of the trace, one observe that these equations only depends on $c+c^T$. 

It remains to show that equations \eqref{eq:HF-2F} and \eqref{eq:EFH} are equivalent to the conditions of the claim; that is, that they can be reduced to \eqref{eq:HF-2F:2} and \eqref{eq:EFH:1}. 

Assuming \eqref{eq:HF-2F:2} and \eqref{eq:EFH:1}, one gets $b_1^{1,3}$ from \eqref{eq:EFH:3}; then, $b_1^{0,5}$ is determined by \eqref{eq:EFH:2}; and, $b_1^{r,r+2}$ is obtained from \eqref{eq:EFH:4}. We claim that \eqref{eq:HF-2F:1} is fulfilled too. Indeed, a long but straightforward computation shows that \eqref{eq:HF-2F:1} is derived from \eqref{eq:coefH}, \eqref{eq:HF-2F:2} together with  the case $r=3$ of \eqref{eq:EFH:4}. Similarly,  \eqref{eq:HF-2F:3} follows from  \eqref{eq:coefH}, \eqref{eq:FyieldsH:3}  and \eqref{eq:EFH:4}.
\end{proof}

\subsection{Extending to ${\mathcal W}_>$}\label{subsect:W>}

In order to extend a map defined on $\sl(2)$ to one on ${\mathcal W}_>$, one should choose an endomorphism $T$, define $\rho(L_i)$ by equations~\eqref{eq:rho(L_2)def} and \eqref{eq:rho(L_i)def} and check infinitely many constraints (see~\cite{PT}). However, in our situation the following Lemma simplifies that approach drastically; there will exist a unique $T$ satisfying all the requirements. 

\begin{lem}\label{lem:adFinjective}
Let $F,H$ be as in equations~\eqref{eq:type-1}   and~\eqref{eq:type0}. The map:
	$$\ad(F):\uh'_i \overset{\sim}\longrightarrow \uh'_{i-1}
	$$
is an isomorphism for $i\geq 2$.
\end{lem}

\begin{proof}
First, one has to prove that given an operator:
	$$
	S\,:=\, b_{i-1}^{0,2i+1} p_{2i+1} \, +\, p_j^T c_{i-1}^{j,2i-j-2} p_{2i-j-2} \, +\, q_j b_{i-1}^{j,j+2i-2} p_{j+2i-2}\,\in\,\uh_{i-1}
	$$
of type $i-1\geq 1$, there is exactly one operator:
	$$
	T\,:=\, b_i^{0,2i+3} p_{2i+3} \, +\, p_j^T c_i^{j,2i-j} p_{2i-j} \, +\, q_j b_i^{j,j+2i} p_{j+2i}\,\in\,\uh_i
	$$
of type $i$ satisfying $\ad(F)(T)=S$ where  $\ad$ denotes the adjoint representation and $F$ is given by equation~\eqref{eq:type-1}.

Now, one proceeds as in the proof of Lemma~\ref{lem:FyieldsH} and shows that  $\ad(F)(T)=[F,T]=S$ has exactly one solution. 


Finally, let us check that if $S\in\uh'_{i-1}$  and $\ad(F)(T)=S$, then $T\in\uh'_{i}$. Using the injectivity of $\ad(F)$ and the relation:
    $$
    \begin{aligned}
    \ad(F)(\ad(H)(T))
    \,& =\,\ad(H)(\ad(F)(T)) + \ad([F,H])(T)
    \,= \\ & =\, \ad(H)(S) + \ad(2F)(T)
    \,=\, 2(i-1)S + 2S\,=\, 2 i S
    \end{aligned}
    $$
one obtains $\ad(H)(T)=2iT$, as we wanted.
\end{proof}

\begin{thm}\label{thm:W^+-sl2}
Let $F$ be as in \eqref{eq:type-1} where  $a $ is symmetric and $b_{-1}^{i,i-2}$ are invertible.

Then, the map $\iota^*$ of \eqref{eq:emb-sl2-W+} yields a bijection:
    {\small
    $$
    \left\{
        \begin{gathered}
        \rho \in \homlie({\mathcal W}_>,\uh)
        \\
         \text{such that }
        \rho(L_{-1})=F
        \\
        \text{ and } 
        \rho(L_i) \in\uh_i \text{ for }i\geq 0 
        \end{gathered}
    \right\}
    \, \overset{\sim}\longrightarrow\,
    \left\{
        \begin{gathered}
        \sigma \in \homlie(\sl(2),\uh) \\ 
         \text{such that }
        \sigma(f)=F\, ,
        \\
        \sigma(h) \in\uh_0 \text{ and }\sigma(e)\in\uh_1
        \end{gathered}
    \right\}
    $$}
\end{thm}

\begin{proof}
Given $\rho$, we define $\sigma:=\iota^*(\rho)$ and, therefore, $\sigma(f)=\rho(\iota(f))=\rho(L_{-1})$, $\sigma(h)=\rho(-2L_0)$ and $\sigma(e)=\rho(-L_1)$.

For the converse, one requires several steps and the previous Lemmas.

\emph{Step 1.} Let  $\sigma$ be given. There exists a $\C$-linear homomorphism $\rho:{\mathcal W}_>\to \uh$ such that $\sigma=\iota^*(\rho)$. First, we set:
    $$
    \rho(L_{-1}):= \sigma(f)=F
    \; , \quad
    \rho(L_0):= -\frac12 \sigma(h)
    \; , \quad
    \rho(L_1):= -\sigma(e)
	$$
The fact that $\sigma$ is a map of Lie algebras and Lemma~\ref{lem:FyieldsH} imply that:
	$$
	\rho(L_0)\,=\, -\frac12H
	$$
where $H$ is as in equation~\eqref{eq:type0}. Furthermore, it holds that $\rho(L_{i})\in\uh'_i$ for $i=-1,1$. Having in mind Lemma~\ref{lem:adFinjective} we obtain that there is a unique $T\in\uh'_2$ such that:	
	$$
	\operatorname{ad}(\rho(L_{-1}))(T) \,=\, \rho(L_1)
	$$

Then, we define:
	\begin{equation}\label{eq:rho(L_2)def}
	\rho(L_2)\,:= \, -3 T \,\in \, \uh'_2
	\end{equation}
and, recursively,
    \begin{equation}\label{eq:rho(L_i)def}
    \rho(L_i):= \frac1{i-2}[\sigma(e),\rho(L_{i-1})]
    \qquad \text{for }i>2
    \;.
    \end{equation}

\emph{Step 2.} It holds that $ [\rho(L_{0}),\rho(L_j)] = -j\rho(L_{j})$ for $j\geq -1$. This is equivalent to show that
$\rho(L_j) \in  \uh'_j$ for all $j\geq 1$.  Bearing in mind that $\rho(L_1)\in\uh'_1$ and Lemma~\ref{lem:T'T'bracket}, the conclusion follows.

\emph{Step 3.} It holds that $ [\rho(L_{-1}),\rho(L_j)] = -(1+j)\rho(L_{j-1})$ for $j\geq -1$. The cases $j\leq 1$ follow from the fact that $\sigma$ is a homomorphism of Lie algebras.  The choice of $T$ implies the case $j=2$. Let us proceed by induction on $j$.  For $j\geq 3$, the definition of $\rho(L_i)$, the Jacobi identity and the induction hypothesis yield:
    {\small $$
    \begin{aligned}\;
    [\rho(L_{-1}), & \rho(L_j)]
    \,  =\,
    [\rho(L_{-1}),\; -\frac1{j-2}[\rho(L_{1}),\rho(L_{j-1})]]
    \, = \\ & = \,
    \frac1{j-2}\Big( [\rho(L_{1}),\; [\rho(L_{j-1}), \rho(L_{-1})]] \;+\;
    [\rho(L_{j-1}),\; [\rho(L_{-1}), \rho(L_{1})]]\Big)
    \, = \\ & = \,
    \frac1{j-2}\Big( [\rho(L_{1}),\; j \rho(L_{j-2})] \;+\;
     [\rho(L_{j-1}),\; (-2) \rho(L_{0})]\Big)
     \, = \\ & = \,
    \frac1{j-2}\big( j(3-j) \rho(L_{j-1}) -2 (j-1)\rho(L_{j-1})\big)
    \,=\, (-1-j)\rho(L_{j-1})
    \end{aligned}
    $$}

\emph{Step 4.} The identity:
    \begin{equation}\label{eq:braket-hom}
    [\rho(L_i),\rho(L_j)]\,-\,(i-j)\rho(L_{i+j}) \,=\,0
    \end{equation}
holds for $i,j\geq 1$. We proceed by induction on $n=i+j$. The case $n=4$ (i.e. $i,j\geq 1$ and $i+j=4$) holds by the very definition of $\rho(L_4)$. Now, let us assume that it holds true up to  $n-1=i+j-1$ and let us prove the case $n=i+j>4$. Observe that, by Step 2, the l.h.s of the equation \eqref{eq:braket-hom} lies in $\uh_{i+j}'$. By Lemma~\ref{lem:adFinjective}, it suffices to show that its image under $\ad(F)=\ad(\rho(L_{-1}))$ vanishes. In fact, the Jacobi identity, the Step 3 and the induction hypothesis show that:
	{\small $$
	\begin{aligned}
	\ad & (\rho( L_{-1}))\big(   [\rho(L_i),\rho(L_j)]\,-\,(i-j)\rho(L_{i+j})\big) \, = \\ & = \,
	 [[\rho(L_{-1}), \rho(L_i)],\rho(L_j)] \,+\,  [\rho(L_i), [\rho(L_{-1}),\rho(L_j)]]
	 \,-\, (i-j)[\rho(L_{-1}),\rho(L_{i+j})]
	 \, = \\ & = \,
	 [-(1+i)\rho(L_{i-1}),\rho(L_j)] \,+\, [\rho(L_i), -(1+j)\rho(L_{j-1})] \,+\, (i-j)(1+i+j) \rho(L_{i+j-1})
	 \, = \\ & = \,
	 - \big( (1+i)(i-j-1) \,+\, (1+j)(i-j+1)  \,-\, (i-j)(1+i+j) \big) \rho(L_{i+j-1})
	 \, = \\ & = \,
	0
	 \end{aligned}
	 $$}

\emph{Step 5.} $\rho$ is a Lie algebra homomorphism. This follows from the properties of $\sigma$ and Steps 2, 3, 4. 	
	
\end{proof}

\subsection{Factorization as a product}

It is remarkable that if the vector space $(A, (\, ,\,))$ decomposes as $A_1\perp A_2$ (i.e. $A=A_1\oplus A_2$ and $(a_1,a_2)=0$ for all $a_i\in A_i$), then the very definition of the associated Heisenberg algebra implies that:
	$$
	{\mathbb H}(A) \,\simeq\, {\mathbb H}(A_1)  \widehat\otimes_{\C} {\mathbb H}(A_2) 
	$$
as Lie algebras and ${\mathbb H}(A_i)$ is a subalgebra of ${\mathbb H}(A)$. So, we may wonder under which circumstances a morphism
$\rho :{\mathcal W}_> \to \uh[(A)]$ would decompose accordingly. The following Theorem provides an answer in terms of the restriction $\rho\vert_{\sl(2)}$. For this goal, recall that matrices $a$, $b$ and $c$'s behave as bilinear forms on $A$ (w.r.t. the action of $\Gl(A)$).

\begin{thm}\label{thm:rhodecomp}
Let $F, H, E$ be as in~\eqref{eq:type-1}, \eqref{eq:type0} and \eqref{eq:E=type1}. Let $\rho :{\mathcal W}_> \to \uh[(A)]$ satisfy $\rho(L_{-1})=F$, $\rho(L_0)=-\frac12 H$ and $\rho(L_1)=-E$. 

If the vector space $A$ decomposes as $A_1\perp A_2$ w.r.t. $\eta$ and this decomposition is compatible with the action of $F$ and with the bilinear forms  $b_0^{1,1}$ and  $c_1^{1,1}$, then there are Lie algebra maps $\rho_{i} :{\mathcal W}_> \to \uh[(A_i)]$ for $i=1,2$ such that:
	$$
	\rho \,=\, \rho_1\otimes 1 + 1\otimes \rho_2
	$$

If this is the case, and $\rho(L_k)\in \uh[(A)]_k'$ for all $k\geq -1$, then $\rho_{i}(L_k)\in \uh[(A_{i})]_k'$ for all $k\geq -1$ and $i=1,2$. 
\end{thm}

\begin{proof}
\emph{Step 1}. The case of $\rho(L_{-1})$. The hypothesis says that we can find $\{a_{\alpha}\vert \alpha=1,\ldots,n\}$, a basis of $A$, and an index $m$ such that, for $1\leq i <m \leq j\leq n$, the vectors $a_i$ and $a_j$ are orthogonal w.r.t. to the bilinear form defined by $a $.  Equivalently, w.r.t. the splitting $A_1\oplus A_2$ the matrix of this bilinear form acquires a block decomposition as follows:
	$$
	a  \,=\, 
	\begin{pmatrix} \ast  & 0 \\ 0 & \ast \end{pmatrix}
	$$
It is now straightforward that the terms of the operator $F$ (as given in equation~\eqref{eq:type-1}) can be grouped in two sets, the first one involving $p_{i,\alpha}$ and $q_{i,\alpha}$ for $i\in{\mathbb N}$ and $1\leq \alpha <m$, and the second one depending only on $p_{i,\alpha}$ and $q_{i,\alpha}$ for $i\in{\mathbb N}$ and $m\leq \alpha \leq n$. Denote these operators as $\bar L_{-1,1}$ and $\bar L_{-1,2}$ respectively. One checks that:
	\begin{equation}\label{eq:barL_-1}
	\begin{gathered}
	\rho( L_{-1}) = \bar L_{-1,1}\otimes 1 + 1\otimes  \bar L_{-1,2}
	\\
	\bar L_{-1,\alpha}\,\in \, \uh[(A_{\alpha})]_{-1}\qquad\alpha=1,2.
	\end{gathered}
	\end{equation}
	
\emph{Step 2}. The case of $\rho(L_{0})$. Bearing in mind that it is defined as $-\frac12H$ and that the coefficients of the latter fulfill the relations \eqref{eq:FyieldsH}, one can proceed as in the previous case. More precisely, considering the following block decompositions:
	$$
		\eta  \,=\, 
	\begin{pmatrix} \eta_1  & 0 \\ 0 & \eta_2 \end{pmatrix}
	\qquad
		a  \,=\, 
	\begin{pmatrix} a_1  & 0 \\ 0 & a_2 \end{pmatrix}
	\qquad
		c_{1}^{1,1} \,=\, 
	\begin{pmatrix} c_1  & 0 \\ 0 & c_2 \end{pmatrix}
	$$
one may use the following identity as a defining relation for $\bar L_{0,\alpha} \in \uh[(A_{\alpha})]_0$:
	$$
	\begin{aligned} 
	\rho(L_0)-b_0^{0,0} 
	\, & =\, \big( \bar L_{0,1}- 2\Tr( c_{1}\eta_1^{-1}( a_{1}+a_1^T)(\eta_1^T)^{-1} )\big) 
	\, + \\ & \qquad +\, 
	 \big( \bar L_{0,2}- 2\Tr( c_{2}\eta_2^{-1} (a_{2}+a_2^T)(\eta_2^T)^{-1})\big)  
	\end{aligned}
	$$

\emph{Step 3}. The case of $\rho(L_{k})$ for $k\geq 1$.  Recall from the proof of Theorem~\ref{thm:sl2-symmmatrices} that the coefficients $b_1^{r,r+2}$ of $\rho(L_{1})$ can be expressed in terms of $a $, $b_0^{1,1}$ and $c_1^{1,1}$ and that a close look of these expressions shows that $b_1^{r,r+2}$ are compatible w.r.t. to the splitting of $A$. Thus, we can express $\rho(L_1)$ as the sum of two factors; namely, $\bar L_{1,\alpha}$ for $\alpha=1,2$ which consists of the terms of $\rho(L_{1})$ in   $p_{i,\alpha}, q_{i,\alpha}$ for  $1\leq \alpha <m$ and  for $m\leq \alpha \leq n$, respectively. Now, we proceed as above. 

For the case of $\rho(L_k)$ for $k\geq 2$ one proceeds recursively (using the expressions of the proof of Lemma~\ref{lem:adFinjective}).

%

\emph{Step 4}. 
$[ \bar L_{k,\alpha} ,  \bar L_{l,\beta}]=0$ for $k,l \geq -1$ and  $\alpha\neq\beta$, since these two operators involve disjoint sets of variables. 

\emph{Step 5}. The maps $\rho_{\alpha}$. Consider:
	$$
	\rho_{\alpha}(L_k)\,:=\,
	\bar L_{k,\alpha} 
	\qquad\text{for $k\geq -1$ and $\alpha=1,2$}
	$$
The previous steps show that $\rho = \rho_1\otimes 1 + 1\otimes \rho_2$.

It remains to check that $\rho_{\alpha}$ are morphisms of Lie algebras. For this goal we will expand both sides of the identity $[\rho(L_k),\rho(L_l)]=(k-l)\rho(L_{k+l})$ using the above facts. The l.h.s. is:
	{$$
	\, [\rho(L_k),\rho(L_l)]\,=\, 
	[\bar L_{k,1}+\bar L_{k,2} , \bar L_{l,1}+\bar L_{l,2}]
	\,=\, 
	[\bar L_{k,1} , \bar L_{k,1}]  + [\bar L_{k,2} , \bar L_{l,2}]
	$$}
while the r.h.s. reads:
	{$$
	(k-l)\rho(L_{k+l}) \,=\, 
	(k-l)( \bar L_{k+l,1}+\bar L_{k+l,2})
	$$}
Comparing both expressions and having in mind the separation of variables, it follows that:
	{$$
	[\bar L_{k,\alpha} , \bar L_{k,\alpha }]  \,=\, 
	(k-l) \bar L_{k+l,\alpha}
	$$}
\noindent and we conclude that $\rho_{\alpha}$ is a map of Lie algebras ${\mathcal W}_> \to \uh[(A_{\alpha})]$.

\emph{Step 6}. Type of the operators. In order to show that $\rho(L_k)\in \uh[(A)]_k'$ implies that  $\rho_{\alpha}(L_k)\in \uh[(A_{\alpha})]_k'$, it suffices to expand the Lie bracket $[ \rho(L_0), \rho(L_k)]$ using  $ \rho(L_k)= \rho_1(L_k) \otimes 1 + 1\otimes \rho_2(L_k)$. 
\end{proof}

\begin{rem}
It is worth noticing that if a decomposition is compatible with $a $, it does not need to be compatible with $b_0^{1,1}$. Indeed, for $A=\C^2$, $\eta=a  = \tiny{\begin{pmatrix} 1 & 0 \\ 0 & 1 \end{pmatrix} }$, the general form of $b_0^{1,1}$ is given by
$\tiny{\begin{pmatrix} -1  & \lambda \\ \lambda  & -1 \end{pmatrix} }$.
\end{rem}

For later use, the following general result will be required.

\begin{thm}\label{thm:solutionrhoi}
Let $\rho_i:{\mathcal W}_> \to \End V_i$, $i=1,2$, be two representations of the Lie algebra ${\mathcal W}_>$. And let us consider the product representation:
	$$
	\rho=\rho_1\otimes 1 +1\otimes \rho_2\, : \, {\mathcal W}_> \longrightarrow \End (V_1\otimes V_2)$$

Let $\sum_{i=1}^r f_{1,i}\otimes f_{2,i} \in V_1 \otimes V_2$.  Assume that $f_{i,1},\ldots,f_{i,r}$ are linearly independent (for $i=1,2$).

It then holds that:
	$$
	\rho(L_k)(\sum_{i}f_{1,i}\otimes f_{2,i})\,=\, 0 \quad \forall k\geq -1
	$$
if and only if:
	$$
	\rho_{i}(L_k)(f_{i,j})\,=0\, \text{ for all $i,j$ and $k\geq -1$ } 
	$$
\end{thm}

\begin{proof}
Let us prove the converse. Bearing in mind that $\rho =\rho_1\otimes 1 + 1\otimes \rho_2$, one can do the following computation:
	{\small $$
	\begin{aligned}
	\rho(L_k)(\sum_i f_{1,i} & \otimes f_{2,i} )\, 
	=\, 
	(\rho_1\otimes 1 + 1 \otimes \rho_2) (L_k)(\sum_i f_{1,i}\otimes f_{2,i} )
	\, \\ & =\, 
	 \sum_i  \rho_1(L_k)(f_{1,i}) \otimes f_{2,i} \, +\, 
	 \sum_i f_{1,i}\otimes \rho_2 (L_k)(f_{2,i}  )
	\,=\, 0
	\end{aligned}
	$$}
and the conclusion follows. 

The direct implication is more subtle. The hypothesis and the decomposition of $\rho$ yield:
	{\small $$
	\begin{aligned}
	0\,=\, \rho(L_k)(\sum_i f_{1,i} & \otimes f_{2,i} )\, 
	=\, 
	(\rho_1\otimes 1 + 1 \otimes \rho_2) (L_k)(\sum_i f_{1,i}\otimes f_{2,i} )
	\, \\ & =\, 
	 \sum_i  \rho_1(L_k)(f_{1,i}) \otimes f_{2,i} \, +\, 
	 \sum_i f_{1,i}\otimes \rho_2 (L_k)(f_{2,i}  )
	\end{aligned}
	$$}
Let $E$ be  the vector space generated by $\{f_{1,1},\ldots, f_{1,r}\}\subset V_1$. Suppose that there exists $l$ such that $\rho_1(f_{1,l})$ does not belong to $E$. Then, let $\chi: V_1 \to \C$ be a linear form such that $\chi(f_{1,i})=0$ for all $i$ and $\chi( \rho_1(f_{1,l}))\neq 0$. Applying $\chi$ to the above equation, one obtains:
	$$
	0\,=\, \sum_i \chi\big( \rho_1(L_k)(f_{1,i}) \big) f_{2,i} \,\in\, V_2
	$$
which contradicts the fact that $f_{2,1},\ldots,f_{2,r}$ are linearly independent. Therefore, it follows that $\rho_1(f_{1,l})$ belongs to $E$ for all $l$ or, equivalently, 
	$$
	\begin{aligned}
	\rho_{1,E}:{\mathcal W}_> \,&\longrightarrow \,\End(E)
	\\
	L_k \quad &\longmapsto \, \rho_{1,E}(L_k):= \rho_1(L_k)\vert_E
	\end{aligned}
	$$
is a Lie algebra homomorphism. Recall that, being ${\mathcal W}_>$ simple, the non-trivial representations of ${\mathcal W}_>$ are faithful. Since $E$ is finite dimensional, $\rho_{1,E}$ must be trivial; that is, $\rho_{1,E}(L_k)=0$ for all $k$. In particular, 
	$$
	0\,=\, \rho_{1,E}(L_k)(f_{1,j}) \,=\,  \rho_1(L_k)(f_{1,j}) \qquad \forall j
	$$
The identities $\rho_2(L_k)(f_{2,j})=0$ are proven similarly.  
\end{proof}

\subsection{Commutation Relations}\label{subsect:CR}

This subsection only collects the explicit computations of some Lie brackets used  previously and, thus, the reader can skip it. From a formal point of view,  we are dealing with generators of $\uh$, $ 1, q_{i,\alpha}, p_{i,\alpha}$,  with $i=1,2,\ldots$ and $\alpha=1,\ldots, n$, that satisfy the following relations:
    $$
    \begin{gathered}
    \, [p_{i,\alpha}, q_{j,\beta}] \,=\, \delta_{i,j} i \eta^{\alpha\beta}\cdot 1 
    \\
   \,  [p_{i,\alpha}, p_{j,\beta}] \,=\,[q_{i,\alpha}, q_{j,\beta}] \,=\, 
    [p_{i,\alpha}, 1] \,=\,[q_{i,\alpha}, 1] \,=\,  0
    \end{gathered}
    $$
and, because of the associativity of composition, we will also use:
	$$
	[a,bc]=\, [a,b] c \,+\, b[a,c]
	$$

We will use the Einstein convention; that is, repeated subindices of the variables $p,q$'s imply the summation is to be done. Recall that $b_i^{0,2i+3}$  and $q_i:=(q_{i,1},\ldots, q_{i,n})$ denote row vectors, $p_i:=(p_{i,1},\ldots, p_{i,n})^T$ are column vectors (the superscript $T$ denotes the transpose), and $a , b_i^{j,2i+j}, c_i^{j,2i-j}$ are $n\times n$ square matrices.

Let us compute some Lie brackets. For instance,
    $$
    \begin{aligned}
    \, & [ b_i^{0,2i+3} p_{2i+3} , b_{j}^{0,2j+3} p_{2j+3}]
    \,=\\ & \qquad  =\,
    [ (b_i^{0,2i+3})_{\alpha} (p_{2i+3})_{\alpha}	\, ,\, (b_{j}^{0,2j+3})_{\beta} (p_{2j+3})_{\beta} ]
    \,=\\ & \qquad  =\,
    (b_i^{0,2i+3})_{\alpha}  [ (p_{2i+3})_{\alpha}	\, ,\,(p_{2j+3})_{\beta} ]  (b_{j}^{0,2j+3})_{\beta} \,=\, 0
    \end{aligned}
    $$
where subindices $\alpha,\beta$ denote the corresponding entries of the vectors. Analogously, we have the following identities:
    \begin{subequations}
    \begin{align*}
	& [ q_1 a  q_1^T ,  b_i^{0,2i+3} p_{2i+3} ] \,=\, 0 \qquad \forall i\geq 0
	\\ &
    [ q_{r} b_{i}^{r,r+2i} p_{r+2i} ,  b_j^{0,2j+3} p_{2j+3}]
    \,=\\ & \qquad  =\,
    - [b_j^{0,2j+3} p_{2j+3} ,  (q_{r})_{\alpha}] (b_{i}^{r,r+2i} p_{r+2i})_{\alpha}
    \\ & \qquad \quad 
    -\, (q_{r} b_{i}^{r,r+2i})_{\alpha} [  b_j^{0,2j+3} p_{2j+3}, (p_{r+2i})_{\alpha}]
    \,=\\ & \qquad  =\,
    - (b_j^{0,2j+3})_{\beta} [ (p_{2j+3})_{\beta}, (q_{r})_{\alpha}]
    (b_{i}^{r,r+2i})_{\alpha\gamma} (p_{r+2i} )_{\gamma}
    \,=\\ & \qquad  =\,
    - (2j+3) (b_j^{0,2j+3})_{\beta} \eta^{\beta\alpha} (b_{i}^{2j+3,2j+3+2i})_{\alpha\gamma} (p_{2j+3+2i} )_{\gamma}
    \,=\\ & \qquad  =\,
    - (2j+3) b_j^{0,2j+3} \eta^{-1} b_{i}^{2j+3,2j+3+2i} p_{2j+3+2i}
     \\ & 
     [ p_{r} c_{i}^{r,2i-r} p_{2i-r} ,  b_j^{0,2j+3} p_{2j+3}]
    \,=\, 0 
    \\ & 
    [ q_1 a  q_1^T , q_{r} b_{j}^{r,r+2j} p_{r+2j} ]
    \,=\, 0 \qquad\forall j\geq 1
       \\ &
    [ q_1 a  q_1^T , q_{r} b_{0}^{r,r} p_{r} ]
    \,=\\ & \qquad  =\,
   (  q_{1} b_{0}^{1,1} )_{\alpha} [ q_1 a  q_1^T , ( p_{1})_{\alpha} ]
    \,=\\ & \qquad  =\,
	- (  q_{1} b_{0}^{1,1} )_{\alpha} \big(
	[( p_{1})_{\alpha} , ( q_{1})_{\beta}] (a  q_1^T)_{\beta} +
	(q_1 a )_{\beta} [( p_{1})_{\alpha} , ( q_{1})_{\beta}]\big)
    \,=\\ & \qquad  =\,
    - q_{1} b_{0}^{1,1}\eta^{-1}  \big( a  +a^T \big) q_1^T
\\ &
    [ q_1 a  q_1^T , p_{r}^T c_{j}^{r,2j-r} p_{2j-r} ]
    \,=\\ & \qquad  =\,
    [ q_1 a  q_1^T , (p_{r})_{\alpha}] (c_{j}^{r,2j-r} p_{2j-r})_{\alpha}
    \,+\,
      (p_{r}^T c_{j}^{r,2j-r})_{\alpha}  [ q_1 a  q_1^T ,  (p_{2j-r} )_{\alpha}]
    \,=\\ & \qquad  =\,
    -\big( [(p_1)_{\alpha}, (q_1)_{\beta}] (a  q_1^T)_{\beta} + (q_1 a )_{\beta}[(p_1)_{\alpha}, (q_1)_{\beta}]\big)
    (c_{j}^{1,2j-1} p_{2j-1})_{\alpha}
    \,- \\ & \qquad \qquad - \, (p_{2j-1}^T c_{j}^{2j-1,1})_{\alpha} \big([(p_1)_{\alpha},(q_1)_{\beta}] (a  q_1^T)_{\beta} + ( q_1 a )_{\beta} [(p_1)_{\alpha},(q_1)_{\beta}]\big)
    \,=\\ & \qquad  =\,
    -q_1 \big(  a    + ( a )^T\big)(\eta^{-1})^T  c_{j}^{1,2j-1} p_{2j-1}     
    \,- \, p_{2j-1}^T c_{j}^{2j-1,1} \eta^{-1} \big(   a   + (a   )^T\big) q_1^T
    \,=\\ & \qquad  =\,
    -q_1 \big(  a    + ( a )^T\big)(\eta^{-1})^T   
    \big(  c_{j}^{1,2j-1}  +    (c_{j}^{2j-1,1})^T \big) p_{2j-1} -
    \\  
    & \qquad 
    -  \delta_{j1}  \operatorname{Tr}\big(c_{1}^{1,1} \eta^{-1} ( a   + a^T)(\eta^{-1})^T\big)
\\ &
    [ q_{r} b_{i}^{r,r+2i} p_{r+2i}   , p_{s}^T c_{j}^{s,2j-s} p_{2j-s} ]
    \,=\\ & \qquad  =\,
    [  q_{r} b_{i}^{r,r+2i} p_{r+2i} , (p_{s})_{\alpha} ] (c_{j}^{s,2j-s} p_{2j-s})_{\alpha}
    \,+
    \\ & \qquad\quad + \,
      (p_{s}^T c_{j}^{s,2j-s} )_{\alpha} [ q_{r} b_{i}^{r,r+2i} p_{r+2i} ,  (p_{2j-s})_{\alpha} ]
    \,=\\ & \qquad  =\,
    - [ (p_{s})_{\alpha} ,  (q_{r})_{\beta}  ] (b_{i}^{r,r+2i} p_{r+2i} )_{\beta} (c_{j}^{s,2j-s} p_{2j-s})_{\alpha}
    \,- \\ & \qquad \qquad -\,
      (p_{s}^T c_{j}^{s,2j-s} )_{\alpha}  [ (p_{2j-s})_{\alpha},  (q_{r})_{\beta}   ](b_{i}^{r,r+2i} p_{r+2i})_{\beta}
    \,=\\ & \qquad  =\,
    - r p_{2j-r}^T (c_{j}^{r,2j-r})^T \eta^{-1} b_{i}^{r,r+2i} p_{r+2i}
    \,-\, r  p_{2j-r}^T c_{j}^{2j-r,r} \eta^{-1} b_{i}^{r,r+2i} p_{r+2i}
    \,=\\ & \qquad  =\,
    - r p_{2j-r}^T \big( (c_{j}^{r,2j-r})^T + c_{j}^{2j-r,r}\big)\eta^{-1} b_{i}^{r,r+2i} p_{r+2i}
      \\ &
    [ q_{r} b_{i}^{r,r+2i} p_{r+2i}   , q_{s} b_{j}^{s,s+2j} p_{s+2j} ]
    \,=\\ & \qquad  =\,
    [  q_{r} b_{i}^{r,r+2i} p_{r+2i} , (q_{s})_{\alpha} ] (b_{j}^{s,s+2j} p_{s+2j})_{\alpha}
    \,+\
     \\ & \qquad \qquad +\,
      (q_{s} b_{j}^{s,s+2j} )_{\alpha} [ q_{r} b_{i}^{r,r+2i} p_{r+2i} ,  (p_{s+2j})_{\alpha} ]
    \,=\\ & \qquad  =\,
    - (q_{r} b_{i}^{r,r+2i})_{\beta}[    (q_{s})_{\alpha}, (p_{r+2i})_{\beta} ] (b_{j}^{s,s+2j} p_{s+2j})_{\alpha}
    \,- \\ & \qquad \qquad -\,
      (q_{s} b_{j}^{s,s+2j} )_{\alpha} [ (p_{s+2j})_{\alpha}, (q_{r} )_{\beta} ](b_{i}^{r,r+2i} p_{r+2i})_{\beta}
    \,=\\ & \qquad  =\,
    (r+2i) q_{r} b_{i}^{r,r+2i}  \eta^{-1}   b_{j}^{r+2i,r+2i+2j} p_{r+2i+2j}
    \,-\,
      r q_{r-2j} b_{j}^{r-2j,r}\eta^{-1} b_{i}^{r,r+2i} p_{r+2i}
   \end{align*}
   \end{subequations}

\section{An Application}

As an application of the previous sections, we offer here an example that illustrates how our results can be used for studying the representation of ${\mathcal W}_>$ appearing in the study of the Virasoro conjecture. Regarding the Virasoro conjecture our main references are the works of  Dubrovin-Zhang, Eguchi-Hori-Xiong, Getzler, Givental and Liu-Tian (\cite{DubrovinZ, EHX, Getzler, Givental, LiuTian}). 

In our example will consider $(A, (,))$ to be the cohomology ring of a smooth projective variety $X$ with $c_1(X)=0$ and trivial odd cohomology groups. Recall that the hypothesis on the first Chern class is equivalent to the vanishing of the operator $R$ in \cite{DubrovinZ, Getzler}; however, it does not seem difficult to extend the results of \S\ref{sec:ReprW>} to include this case. On the other hand, the hypothesis on the odd cohomology groups is fulfilled if $X$ has generically semisimple even quantum cohomology (\cite{HMT}). It seems to be very hard to weaken this assumption.

\subsection{Preliminaries}

Let $A$ be a $n$ dimensional vector space over $\C$ endowed with a bilinear form $(\, , \, )$. Let $\{a_1,\ldots,a_n\}$ be a basis and $\eta$ be the  matrix associated to the pairing,  $\eta_{\alpha\beta}:=(a_{\alpha},a_{\beta})$. Let us  consider  the subspace $\C[[t_1,t_3,t_5,\ldots]]$ of the the boson Fock space $\C[[t_1,t_2,\ldots]]$ and the subalgebra of $ \C[[t_1,\ldots]] \hat\otimes_{\C} S^{\bullet}A$ generated by $t_{i,\alpha}:= t_i\otimes a_{\alpha} $ with $i$ odd:
	{\small
	\begin{equation}
	\label{eq:Vkdv}
	\Vkdv(A)\,:=\C[[\{t_{i,\alpha} \vert 1\leq \alpha\leq n, i \text{ odd}\} ]]
	\,\subseteq \,  \C[[t_1,\ldots]] \hat\otimes_{\C} S^{\bullet}A
	\end{equation}}
If no confusion arises, we will simply write $\Vkdv$. 

Now we study a distinguished representation of ${\mathcal W}_>$ in $\Vkdv$; eventually, we will see that it is the representation coming from the action of the Heisenberg algebra via Givental's quantization (\cite{Givental}). More precisely, we will combine the chain of inclusions of Lie algebras:
	$$
	\sl(2)\,\hookrightarrow\, {\mathcal W}_> \, \hookrightarrow \, \uh
	$$
which has been studied in the previous section, with a map:
	$$
	\begin{aligned}
	\widehat{\,}\,: \uh[(A)] &\, \longrightarrow \End_{\C}(\Vkdv(A))
	\\
	P &\, \longmapsto \, \hat P
	\end{aligned}
	$$
whose obstruction to be compatible with the Lie brackets is governed by a cocycle. This map is defined following  the results of Dubrovin-Zhang, Givental and Kazarian (\cite{DubrovinZ, Givental, Kaz}); namely, we set:
	\begin{equation}\label{eq:quant}
\hat 1 \,=\, 1 \, , \qquad 
\hat p_{i,\alpha} \,=\, \eta^{\alpha\beta}\frac{\partial}{\partial t_{i,\beta}}   \, ,\qquad
\hat q_{i,\alpha} \,=\,  i t_{i,\alpha}
	\end{equation}
(recall that $i$ is a positive odd integer number).

\begin{rem}
Givental has developed a beautiful formalism for this construction in terms of quantization of quadratic hamiltonians (\cite{Givental}). An alternative approach, originated in the Japanese school and strongly linked to the Sato grassmannian, can be found in \cite{KSU}. The forthcoming section (\S\ref{subsec:solutions1dim}) is deeply inspired by the latter. 
\end{rem}

\subsection{The representation}\label{subsec:representation}

Bearing in mind the results of \S\ref{subsect:sl2}, we know that the operator:
	$$
	    F\,  :=\, b_{-1}^{0,1} p_1 + q_1\eta q_1^T + q_{i+2} \eta p_i 
	$$
together with the data:
	\begin{itemize}
	\item  $b_{-1}^{0,1}$ arbitrary,
	\item $b_0^{1,1}$ such that \eqref{eq:FyieldsH:2} holds, and
	\item  $c_1^{1,1} := \frac1{16} \eta^T b_0^{1,1} \eta^{-1} (b_0^{1,1}\eta^{-1}+2) $,
	\end{itemize}
determine a map  $\sigma:\sl(2)\to \uh$. Indeed, equations~\eqref{eq:FyieldsH}, \eqref{eq:EFH} and \eqref{eq:HF-2F} allow us to obtain the explicit expressions for  $H$ and $F$:
	\begin{equation}\label{eq:HEparticular}
	\begin{aligned}
	H\, & = \, 
	\frac13 b_{-1}^{0,1}(\eta^{-1} b_0^{1,1}-2) p_3 +
	 \frac1i q_i (b_0^{1,1}-(i-1)\eta) p_i +
	 \\ & \qquad + 
	 \frac1{16} \Tr(b_0^{1,1}\eta^{-1} (b_0^{1,1}\eta^{-1}+2) (1+\eta^{-1}\eta^T))
	\\
	E \, &=\, \frac1{5!!} b_{-1}^{0,1} 
	\big(2\eta^{-1}b_0^{1,1} - 2- (\eta^{-1}+(\eta^{-1})^T )(c_1^{1,1}+(c_1^{1,1})^T)\big) p_5 
	+ \\
	 & \qquad +
	 \frac1{16} p_1^T \eta^T b_0^{1,1} \eta^{-1} (b_0^{1,1}\eta^{-1}+2)  p_1 - \\
	 & \qquad - \frac1{4 i (i+2) } q_{i} (  b_0^{1,1} \eta^{-1}- (i-1) ) (  b_0^{1,1} \eta^{-1}-(i+1) ) \eta  p_{i+2}
	\end{aligned}
	\end{equation}

Now, by  Theorem \ref{thm:W^+-sl2}, the map $\sigma$ extends uniquely to an homomorphism $\rho:{\mathcal W}_>\to \uh$. And one can now compute the induced action on $\Vkdv$. Let us write down the first operators:
    \begin{subequations}
    \begin{align*}
    \hat{L}_{-1}\, & :=\, (\rho(L_{-1}))^{\hat\,} \, =\, \hat{F}\, =\, b_{-1}^{0,1} \eta^{-1}\frac{\partial}{\partial t_{1}}
    	 + t_{1} \eta t_{1}^T + (i+2) t_{i+2} \frac{\partial}{\partial t_{i}}
    \\
     \hat{L}_{0}\, & :=\, (\rho(L_{0}))^{\hat\,} \, =\, -\frac12 \hat{H}\, =\, 
     - \frac16 b_{-1}^{0,1}(\eta^{-1} b_0^{1,1}-2) \eta^{-1} \frac{\partial}{\partial t_{3}}  -
      \\ & \qquad 	-\frac12 t_i (b_0^{1,1}\eta^{-1} -(i-1)) \frac{\partial}{\partial t_{i}}
      -\frac1{32} \Tr(b_0^{1,1}\eta^{-1} (b_0^{1,1}\eta^{-1}+2) (1+\eta^{-1}\eta^T))
     \\
     \hat{L}_{1}\, & :=\, (\rho(L_{1}))^{\hat\,} \, =\, - \hat{E}\, = \\ & \qquad
     - \frac1{5!!} b_{-1}^{0,1} 
	\big(2\eta^{-1}b_0^{1,1} - 2- (\eta^{-1}+(\eta^{-1})^T )(c_1^{1,1}+(c_1^{1,1})^T)\big) \eta^{-1}\frac{\partial}{\partial t_{5}} - \\ & \qquad
     -  \frac1{16} (\frac{\partial}{\partial t_{1}})^T \eta^T  b_0^{1,1} \eta^{-1} (b_0^{1,1}\eta^{-1}+2)   \frac{\partial}{\partial t_{1}} + \\
	 & \qquad + \frac1{4 (i+2)} t_{i} (  b_0^{1,1} \eta^{-1}- (i-1) ) (  b_0^{1,1} \eta^{-1}-(i+1) )  \frac{\partial}{\partial t_{i+2}}
     \end{align*}
     \end{subequations}
where, as usual, we write $t_i$ for the row vector $(t_{i,1},\ldots, t_{i,n})$ and $\frac{\partial}{\partial t_{i}}$ for the column vector $(\frac{\partial}{\partial t_{i,1}},\ldots, \frac{\partial}{\partial t_{i,n}})^T$.

\subsection{The operators of the Virasoro Conjecture: a baby model}\label{subsec:baby}

Now, we are ready to show how the operators appearing in the Virasoro conjecture agree with our approach for the case of manifold with trivial odd cohomology and whose first Chern class vanishes. 

From now on, we suppose we are given $X$, whose first Chern class is zero, and with trivial odd cohomology. Under this hypothesis, the Poincar\'e pairing defines on $A:= H^{\bullet}(X,{\mathbb C})$  a symmetric non-degenerated bilinear form:
	$$
	( a, b ) \,=\, \int_X a\cup b \quad \text{ for }a,b\in A ,
	$$
Let $r:=\operatorname{dim}(X)$ and fix a basis $\{a_{\alpha}\vert \alpha=1,\ldots, n \}$ of $A$,  with $a_1=1\in H^0(X,{\mathbb C})$, such that it is homogeneous w.r.t. the Hodge decomposition; that is, $a_{\alpha}\in H^{p_{\alpha},q_{\alpha}}(X)$ for certain $p_{\alpha}, q_{\alpha}$.  Let $\bar \eta$ the matrix associated to the Poincar\'e pairing w.r.t. the chosen basis and let us define  $\mu_{\alpha}:= p_{\alpha}-\frac{r}2$ and $\mu$ the matrix with $\mu_1,\ldots, \mu_r$ along its diagonal and $0$ elsewhere. Observe that the compatibility of the Poincar\'e pairing w.r.t. the Hodge decomposition yields:
	\begin{equation}\label{eq:mua+mb=0}
	\bar\eta_{\alpha\beta}\neq 0 \qquad \implies \qquad \mu_{\alpha} + \mu_{\beta}=0
	\end{equation}
%

The operators appearing in the Virasoro conjecture when the first Chern class vanishes (\cite[Equation (1.2)]{Getzler}) are as follows:
	{\small
	{\begin{equation}\label{eq:Getzler-101}
	\begin{aligned}
	\bar L_{-1}\,& :=\,
	- \frac{\partial}{\partial \bar t_{0,1}}  + \frac{1}{2\hbar} \bar t_{0} \bar \eta \bar t_0^T + \bar t_{i+1} \frac{\partial}{\partial \bar t_{i}} 
	\\
	\bar L_0 \,& :=\, -\frac{3-r}2\frac{\partial}{\partial  \bar t_{1,1}}  + (\mu_{\alpha}+i+ \frac{1}2) \bar t_{i,\alpha}  \frac{\partial}{\partial \bar t_{i,\alpha}}  +  
	\frac{1}{48}(3-r)\int_X c_r(X)
	\end{aligned}
	\end{equation}}}
and, for $k\geq 1$, as:
	{\small
	\begin{equation}\label{eq:GetzlerLk}
	\begin{aligned}
	\bar L_k \,& :=\, -\frac{\Gamma(k+\frac{5-r}2)}{\Gamma(\frac{3-r}2)} 
	\frac{\partial}{\partial \bar t_{k+1,1}} +
		 \frac{\Gamma(\mu_{\alpha}+i+k+\frac{3}2)}{\Gamma(\mu_{\alpha}+i+\frac{1}2)} \bar t_{i,\alpha }   \frac{\partial}{\partial \bar t_{k+i,\alpha}} 
		\\ & \qquad
		+ \frac{\hbar}2 (-1)^{i} \frac{\Gamma(\mu_{\alpha}+i+k+\frac{3}2)}{\Gamma(\mu_{\alpha}+i+\frac{1}2)} \bar \eta^{\alpha\beta} \frac{\partial}{\partial \bar t_{-1-i,\alpha}}  \frac{\partial}{\partial \bar t_{k+i,\beta}} 
	\end{aligned}
	\end{equation}}
where $c_r(X)$ is the $r$-th Chern class and we have used variables $\bar t_{i,\alpha}$ with $\alpha=1,\ldots, n$ and $i=0,1,2,\ldots$.

Similarly to the case of $\uh$, we say that a second order differential operator in $\{\bar t_{i,\alpha}\}$ is of type $i$ if it is a linear combination of $\frac{\partial}{\partial \bar t_{i+1,\alpha}}$ and the following terms 
$\frac{\partial}{\partial \bar t_{j-1,\alpha}}\frac{\partial}{\partial \bar t_{i-j,\beta}}$, $\bar t_{j,\alpha} \frac{\partial}{\partial \bar t_{j+i,\beta}}$ and $\bar t_{j-1,\alpha}\bar t_{-i-j,\beta}$ and, if $i=0$, a constant term. Observe that $\bar L_k$ is of type $k$. Now, we offer a simple proof of a folk statement.

\begin{prop}
The operators $\{\bar L_k\vert k\geq 2\}$ are uniquely determined by $\{\bar L_{-1},\bar L_0, \bar L_1\}$ and the condition that $\bar L_k$ is of type $k$ for all $k\geq -1$. 
\end{prop}

\begin{proof}
Under the change of variables $\bar t_i:= \sqrt{2\hbar} (2i+1)!! t_{2i+1}$, it is clear that a second order differential operator in $\bar t_i$'s  is of type $k$ if and only if is equal to $\hat T$ for $T\in \uh_{k}$. Now, it is easy to check that the hypothesis of Theorem \ref{thm:W^+-sl2} hold; namely, $\hat F=\bar L_{-1}$ and $\bar L_k$ are of type $k$ for $k=0,1$. The conclusion follows. 
\end{proof}

\begin{thm}\label{thm:Lbar=Lhat}
It holds:
	$$
	\bar L_i \,=\, \hat L_i \qquad i=-1,0,1,\ldots
	$$
for the choice $\bar t_i:= \sqrt{2\hbar} (2i+1)!! t_{2i+1}$, $\eta=\bar \eta$, $b_{-1}^{0,1}= (0,\ldots, 0, \frac{-1}{\sqrt{2\hbar}})$ and:
	$$
	b_0^{1,1}\,:=\, -(2 \mu+1) \eta \,=\, -{\small \begin{pmatrix} 0 & & 2\mu_1+1 \\ & \iddots & \\ 2\mu_n+1 & & 0\end{pmatrix}}$$
\end{thm}

\begin{proof}
Theorem \ref{thm:W^+-sl2} implies that it suffices to show that $\bar L_i \,=\, \hat L_i$ for $i=-1,0,1$. Indeed, this fact follows from the explicit substitution of $t_i$, $\eta$, etc. as in the statement in the operators $ \hat L_i$. The only identity which is not obvious is the one corresponding to the constant term of $ \hat L_0$. Bearing in mind the definitions and the fact that $\eta$ is symmetric, this term is:
	{\small $$
	\begin{aligned}
	 -\frac1{32} & \Tr (b_0^{1,1}\eta^{-1} (b_0^{1,1}\eta^{-1}+2) (1+\eta^{-1}\eta^T))
	 \,= \\ & = \,
	 -\frac1{16}
	 \sum_{\alpha=1}^n (2 p_{\alpha} -r+1) (2 p_{\alpha}-r -1)  
	  \,=\,
	 \frac14 \sum_{p,q} h^{p,q} (\frac{r+1}2 - p ) ( p-\frac{r-1}2)
	 \end{aligned}
	 $$}
where $h^{p,q}=\dim H^p(X,\Omega^q)$.

Now, observe that the Libgober-Wood identity (\cite[Proposition~2.3]{LW}) can be stated as:
	{\small $$
	\sum_{p,q} (-1)^{p+q}h^{p,q}  (\frac{r+1}2 - p ) ( p-\frac{r-1}2) 
	\,=\, \frac16\int_X \big(\frac{3-r}2 c_r(X)- c_1(X)c_{r-1}(X)\big)
	$$}
Recalling that we are assuming that that it has trivial odd cohomology, the constant term equals:
	$$
	 \frac1{48}\int_X \big( (3-r) c_r(X)- 2 c_1(X)c_{r-1}(X)\big)
	$$
which agrees with the free term of $\bar L_0$ (see \eqref{eq:Getzler-101}) since $c_1(X)=0$. 
\end{proof}

\begin{rem}
It is worth noticing that up to rescaling the variables and a Dilaton shift, these operators coincide with those of \cite[Equation (3.5)]{DVV} and \cite[\S3]{Givental} (for $b_1=0$) and with those of \cite[Equation (7.33)]{Dijkgraaf} and \cite[Equation (2.59)]{Witten} (for $b_1=-\frac13\sqrt{\frac{\eta}{2\hbar}}$). 
\end{rem}

Now, we will go one step further in the study of the above representation. Recall that in \S\ref{subsec:prelim-beta-Diff-to-C[[t]]} it was stated that matrices $a$, $b_{i}^{j-2i,j}$ and $c_i^{j,2i-j}$ behave as bilinear forms under the action of $\Gl(A)$. \emph{A fundamental observation is that all results and equations above are invariant under the action of the general linear group} (acting as base changes on the given basis $\{a_1,\ldots, a_n\}$). Let us briefly discuss this statement. For instance, let $S\in \Gl(A)$, then the row vector $q_i=(q_{i,1},\ldots, q_{i,n})$ is transformed to $q_i S^T$, accordingly the column vector $p_i$ goes to $S p_i$. The action of $S$ sends the bilinear form $\eta$  to $(S^{-1})^T \eta S^{-1}$ and analogously with $a$, etc.~. Note that, since $\eta$ and $b_0^{1,1}$ behave as bilinear forms, $\eta^{-1}b_0^{1,1}$ defines an endomorphism of $A$.  Finally, the Heisenberg algebra is also affected.

\begin{defn}
Let ${\mathbb H}^{\eta}$ be the Heisenberg algebra defined in \eqref{eq:pqcommutation}. Given a map of Lie algebras $\rho:{\mathcal W}_>\to {\mathcal U}({\mathbb H}^{\eta})$ and  $S\in\Gl(A)$,  we denote by $\rho^S$ the map of Lie algebras:
	$$
	{\mathcal W}_>\, \overset{\rho}\longrightarrow \,  {\mathcal U}({\mathbb H}^{\eta})
	\,\overset{\sim}\longrightarrow\,  {\mathcal U}({\mathbb H}^{(S^{-1})^T \eta S^{-1}})
	$$
where the last map sends 	$q_{i}$ to $q_{i} S^T$ and $p_i$ to $S p_i$. 
\end{defn}

With the hypothesis and choices of above, we have the following,

\begin{thm}\label{thm:Virasorodecomp}
Let $\rho:{\mathcal W}_>\to {\mathcal U}({\mathbb H}^{\eta})$ be as above; i.e. $\hat\rho$ defines  the Virasoro constraints,  \eqref{eq:Getzler-101} and \eqref{eq:GetzlerLk},  of a smooth projective variety with trivial odd cohomology and vanishing first Chern class. 

Then  there exists $S\in \Gl(A)$ such that  $\rho^S$ decomposes as the product of $n$ representations of dimension $1$; that is, there exist $\rho_i: {\mathcal W}_>\to \uh[(\C)]$ such that:
	\begin{equation}\label{eq:rho=sumrhoalpha}
	\rho^S \,=\, 
	\rho_1 \otimes 1\otimes\ldots \otimes 1 \, +\, \ldots\, +\,
	1\otimes\ldots \otimes 1 \otimes \rho_n
	\end{equation}
\end{thm}

\begin{proof}
Let us consider a basis  which is orthonormal for $\eta$. Let $S\in\Gl(A)$ be the matrix associated to this change of basis. Due to the choices of $a$, $b_{-1}^{i+2,i}, b_0^{1,1}$ and $c_1^{1,1}$, it is trivial that $S$ also bring them into diagonal form; or, equivalently, there is a common orthogonal basis for all these bilinear pairings. Applying Theorem \ref{thm:rhodecomp}, one concludes.
\end{proof}

In this situation, for each $\alpha=1,\ldots,n$, one obtains a one dimensional representation $\rho_{\alpha}$ or, what is tantamount, our study essentially reduces to the case of Example~\ref{exam:dim1}. That is,  $\dim A=1$ , $a= \eta\in\C^{\ast}$ and, thus, $b_0^{1,1}=-\eta$. Setting $b_0:= b_{0}^{0,0}$, one has that \eqref{eq:HEparticular} gives:
	$$
	\begin{aligned}
	F \,& =\,  b_{-1}  p_1 +   q_1  \eta  q_1  +  q_{i+2}\eta p_i
	\\
	H\, &=\,  -  b_{-1}   p_3   -  q_i  \eta  p_i  - \frac18 
    	\\
	E\,&=\, - \frac1{4} b_{-1}  p_5 - \frac1{4} p_1 \eta  p_1 -  \frac1{4}  q_{i}  \eta  p_{i+2}
	\end{aligned}
	$$
where $b_{-1} $ and $\eta$ are computed from the $n$-dimensional setup \eqref{eq:Getzler-101}.

These three operators determine $\rho$ completely and, according to the map \eqref{eq:quant} and Theorem~\ref{thm:Lbar=Lhat}, one has:
	{\small
	{\begin{equation}\label{eq:barL}
		\begin{aligned}
	\bar L_{-1}\,& :=\, 
	b_{-1} \sqrt{2\hbar} \eta^{-1} \frac{\partial}{\partial \bar t_{0}}  
	+ \frac{1}{2\hbar} \eta \bar t_{0}^2
	+ \bar t_{i+1} \frac{\partial}{\partial \bar t_{i}} 
	\\
	\bar L_0 \,& :=\, \frac32  b_{-1} \sqrt{2\hbar} \eta^{-1} \frac{\partial}{\partial \bar t_{1}}  
	+ (i+\frac12) \bar t_{i}  \frac{\partial}{\partial \bar t_{i}}  + \frac{1}{16} 
	\\
	\bar L_1 \,& :=\, \frac{5!!}4  b_{-1} \sqrt{2\hbar} \eta^{-1} \frac{\partial}{\partial \bar t_{2}}  
	+ \frac{\hbar}2\eta^{-1} \frac{\partial}{\partial \bar t_0} \frac{\partial}{\partial \bar t_0} 
	+ (i+\frac12)(i+\frac32) \bar t_{i } \frac{\partial}{\partial \bar t_{i+1}} 
	\end{aligned}
	\end{equation}}}


\subsection{On the solutions for the $1$-dimensional situation}\label{subsec:solutions1dim}
 Once the representation has been decomposed in terms of $1$-dimensional parts, we wonder if one could deduce some properties of the solutions of the Virasoro constraints. Our approach follows closely our previous work \cite{Plaza-AlgSol} which is inspired in \cite{KacSchwarz}. Briefly, the idea is to show that each of our representations $\rho_i$ come from an action of ${\mathcal W}_>$ on the Sato Grassmannian and that that they admit exactly one solution $\tau_i$, which are $\tau$-functions for the KP hierarchy and, then, conclude that the product $\tau_1\cdot\ldots\cdot \tau_n$ is a solution for $\rho^S$. 

Let us begin recalling that the Sato Grassmannian is the set of subspaces $U\subset \C((z))$ such that the kernel and cokernel of $\pi_U:U\to \C((z))/\C[[z]]$ are finite dimensional (\cite{Sato, SW}). Actually, it is an infinite dimensional scheme (\cite{AMP}) and carries a distinguished line bundle, the determinant line bundle ${\mathbb D}$. Each integer $n$ correspondes to a connected component, $\Gr^n$; namely, those subspaces $U$ such that $\dim\ker\pi_U-\dim\coker\pi_U=n$. Sato-Sato's achievement was to show that there was a bijection between the set of those $U$ s.t. $\pi_U$ is an isomorphism and the set of functions $\tau(t)\in\C[t_1,t_2,\ldots]]$ with $\tau(0)=1$ and fulfilling the KP hierarchy (thus, each $U$ has a $\tau$-function; see \cite{Sato,SW,AMP} for details). The same holds for the Sato grassmannian of $\C((z))^{\oplus n}$ and the $n$-multicomponent KP hierarchy.

The fact that the space of global section of ${\mathbb D}^*$ is isomorphic to the semi-infinite wedge product or Fermion Fock space:
	{\small $$
	H^0(\Gr^n,{\mathbb D}^*)\simeq \wedge^{\frac{\infty}2}\C((z)) \, =\,
	<\left\{ 
	\begin{gathered}
	z^{i_1}\wedge z^{i_2}\wedge\ldots \text{ s.t. } i_1<i_2<\ldots \\
	 \text{ and } i_k=k+n \,\,  \forall k>>0
	 \end{gathered}\right\}>
	 $$}
have allowed its extensive use in CFT's (in particular, by the Japanese school, see \cite{KSU} and references therein). Recall that the boson-fermion correspondence is the isomorphism (we restrict us to $\Gr^0$; that is, the charge $0$ sector):
	$$
	 \wedge^{\frac{\infty}2}\C((z)) \, \simeq \, \C[[t_1,t_2,\ldots]]
	 $$
that maps $z^{i_1}\wedge z^{i_2}\wedge\ldots$ to the Schur polynomial associated with the partition $1-i_1\geq 2-i_2\geq \ldots$. Similarly, the space of global sections of ${\mathbb D}^*$ over the Sato grassmannian of $\C((z))^{\oplus n}$ is isomorphic to $\C[[\{t_{i,\alpha} \vert \alpha=1,\ldots, n , i=1,2,\ldots \}]]$.

Given a subgroup of the restricted linear group of $\C((z))$ (see \cite{SW}), one has an induced action on $\Gr^n(\C((z)))$. Moreover, if the action preserves the determinant bundle, it will yield a projective action on the space of global sections. In fact, an analogous statement holds for the case of Lie algebras. Let us illustrate this issue with the case of the Lie algebra $\operatorname{Diff}^1(\C((z)))$ of first order differential operators on $\C((z))$. An operator $D\in\operatorname{Diff}^1(\C((z)))$ acts on sections as follows. If the matrix $(d_{ij})$ corresponding to $D$ w.r.t. the basis $\{z^i\}$ has no non-trivial diagonal elements, then:
    $$
    D ( z^{i_1}\wedge z^{i_2}\wedge\ldots)
    \,:=\,
     D ( z^{i_1})\wedge z^{i_2}\wedge\ldots +
     z^{i_1}\wedge  D ( z^{i_2})\wedge\ldots + \ldots
     $$
If the matrix $(d_{ij})$ is diagonal, then:
    $$
    D ( z^{i_1}\wedge z^{i_2}\wedge\ldots)
    \,:=\,
     \sum_{j=1}^{\infty} (d_{i_j i_j}- d_{jj})  z^{i_1}\wedge z^{i_2}\wedge\ldots
     $$
Having in mind the boson-fermion correspondence, the above construction gives rise to a linear map:
    \begin{equation}\label{eq:betaDiffEnd}
    \begin{aligned}
    \operatorname{Diff}^1(\C((z))) \,&
    \overset{\beta}\longrightarrow
    \End(\C[[t_1,t_2,\ldots]])
    \\
    D\,& \longmapsto \, \beta( D)
    \end{aligned}
    \end{equation}
which defines a projective representation. Note, nevertheless, that if we are given a map of Lie algebras $\sigma:{\mathcal W}_>\to\operatorname{Diff}^1(\C((z)))$, then, $\beta\circ\sigma$ can be canonically promoted to a linear representation since ${\mathcal W}_>$ has no non-trivial central extensions. Indeed, for this goal, if suffices to add a constant to $\beta\circ\sigma(L_0)$.

The following results will show that the operators of \S\ref{subsec:baby} arise from the previous setup. 

\begin{lem}\label{lem:DinDiffistypei}
Let $D\in\operatorname{Diff}^1(\C((z)))$. Then, $\beta(D)$  is of type $i$ if and only if $D$ is a linear combination of $ 1$, $ z^{-(2i+3)} $ and $  z^{-2i}(z\partial_z+\frac{1-2i}2)$. 
\end{lem}

 \begin{proof}
Recall that $\operatorname{Diff}^1(\C((z)))$ is generated as $\C$-vector space by $1$, $z^m$ for $m\in\Z$ acting as an homothety and $ z^{m}(z\partial_z+\frac{m+1}2)$ for $m\in\Z$. Let us recall from~\cite[Table 1]{Kaz} the description of  the operators induced by them via the boson-fermion correspondence: 
    $$
    \beta(z^m) \,=\,
        \begin{cases}
        m t_m & \text{ for } m >0 \\
        0 & \text{ for } m=0 \\
        \frac{\partial}{\partial t_{-m}} &\text{ for } m<0
        \end{cases}
    $$
and, for $m>0$, 
    {\small $$
    \beta\big(z^m(z\partial_z+\frac{1+m}2)\big)
    \,  = \,
    \frac12 \sum_{j=1}^{m-1} j(m-j) t_j t_{m-j} +
    \sum_{j=1}^{\infty} (j+m) t_{m+j} \frac{\partial}{\partial t_{j}}
    $$}
Analogously, the
action of $z^{-m}(z\partial_z+\frac{1-m}2)$ on $\C((z))$
corresponds to the action of:
    {\small $$
	\beta\big(z^{-m}(z\partial_z+\frac{1-m}2))\big)
    \, = \,
    \sum_{j=1}^{\infty} j t_j  \frac{\partial}{\partial t_{m+j}}  +
    \frac12\sum_{j=1}^{m-1} \frac{\partial}{\partial t_{j}} \frac{\partial}{\partial t_{m-j}}
    $$}
Finally, recall that the case $m=0$ is regularized as follows:
    {\small $$
    \beta\big(z^{-m}(z\partial_z+\frac{1-m}2))\big)    \, := \,  \sum_{j=1}^{\infty} j t_j \frac{\partial}{\partial t_{-j}}
    $$}
Checking the degrees, the conclusion follows.
\end{proof}

\begin{lem}\label{lem:sigma}
Let $\sigma\in\homlie({\mathcal W}_>,\operatorname{Diff}^1(\C((z))))$. Recall that $\Vkdv = \C[[t_1,t_3,\ldots]]$. 

$\beta({\sigma(L_i)}) \vert_{\Vkdv}$ takes values in $\Vkdv$ and it is of type $i$ for all $i$, if and only if there exists $s,  t\in \C$ such that:
    $$
    \sigma(L_i)\,=\,
    t^i\Big(
        \frac12 z^{-2i}(z\partial_z + \frac{1-2i}{2})+ s z^{-2i-3}\Big)
        \qquad \forall i\geq -1
    $$
\end{lem}

\begin{proof}
The ``if'' part follows from Lemma \ref{lem:DinDiffistypei} and the fact that $\sigma$ as in the statement defines a map of Lie algebras. Let us now deal with the ``only if'' part.

We know from~\cite[\S2]{Plaza-AlgSol} (see also~\cite{Plaza-Opers}) that there is a 1-1-correspondence:
    {\footnotesize
    $$ 
    \left\{\begin{gathered}
    \sigma\in\Hom_{\text{Lie-alg}}({\mathcal
    W}_>,\operatorname{Diff}^1(\C((z))))
   \\
    \text{such that }\sigma\neq 0
    \end{gathered}\right\}
    \, \overset{1-1}\longleftrightarrow \, 
    \left\{\begin{gathered}
    \text{triples }(h(z),c,b(z))\text{ such that}
    \\
    h'(z)\in\C((z))^*,\; c\in\C,\; b(z)\in\C((z))
    \end{gathered}\right\}
    $$ 
    }
which is explicitly given by:
    \begin{equation}\label{eq:explicitexpreL_idiffC((z))}
    \sigma(L_i) \,=\,
    \frac{-h(z)^{i+1}}{h'(z)} \partial_z  -
    (i+1) c \cdot h(z)^i + \frac{h(z)^{i+1}}{h'(z)}b(z)
    \end{equation}

On the other hand, due to Lemma~\ref{lem:DinDiffistypei}, the fact that $\sigma(L_i)$ is of type $i$ implies that there exist $r_i,s_i,t_i\in\C$ satisfying:
    \begin{equation}\label{eq:explicitexpreL_ilinearcomb}
    \sigma(L_i) \,=\,
    r_i\cdot 1 + s_i\cdot z^{-(2i+3)} + t_i\cdot z^{-2i}(z\partial_z+\frac{1-2i}2)
    \end{equation}

Comparing the coefficients of $\partial_z$ in the previous identities, it follows that $h(z)=\frac{t_i}{t_{i-1}} z^{-2}$. Hence, the quotients $\frac{t_i}{t_{i-1}}$ are all equal to a constant, say $t$. Hence $t_i=t^i t_0$  and $h(z)=t z^{-2}$. Further, the case $i=0$ yields $t_0=\frac12$.

Pluging this in equations~\eqref{eq:explicitexpreL_idiffC((z))} and~\eqref{eq:explicitexpreL_ilinearcomb}, one gets:
    {\small
    $$
    -(i+1)c(tz^{-2})^i
    -\frac{(tz^{-2})^{i+1}}{2 t z^{-3}} b(z)
    \,=\,
    r_i + s_i z^{-(2i+3)} + \frac12 t^i z^{-2i}(\frac{1-2i}2)
    $$}
and, thus:
    $$
    b(z)\,=\,
    -2(i+1)c z^{-1}
    -  2 t^{-i} r_i z^{2i-1} - 2 s_i t^{-i} z^{-4} - \frac12  z^{-1}(1-2i)
    $$
Observe that the l.h.s. does not depend on $i$, one gets many conditions. First, for $i\neq 0$ the term $z^{2i-1}$ is an odd power of $z$ different from $z^{-1}$ that can not be cancelled with any other term; consequently, $r_i=0$ for $i\neq 0$. Further, since the coefficient of $z^{-4}$ in $b(z)$ has to be independent of $i$, it follows that $t^{-i} s_i$ is a constant independent of $i$; and, thus, equal to $s_0$. Finally, the coefficient of $z^{-1}$ in $b(z)$ is:
    $$
    -2(i+1)c - 2 r_0 \delta_{i,0} - \frac12 (1-2i)
    $$
Since it has to be independent of $i$, it follows that $c=\frac12 $, $r_0=0$ and, thus:
    $$
    b(z)\,=\, -\frac32  z^{-1} - 2 s_0 z^{-4}
    $$
Substituting $h(z), c, b(z)$ into expression~\eqref{eq:explicitexpreL_idiffC((z))} and setting $s=s_0$, one obtains the result.
\end{proof}

Let us recall that the rescaling of the variables yields an action on the boson Fock space. More precisely, $\lambda =\{\lambda_i \} \in\prod_{i \text{ odd}}\C^*$ maps $t_i$ to $\lambda_i t_i$. Accordingly, it acts on $ \homlie({\mathcal W}_>,\End(\Vkdv))$ and sends $\rho$ to $\rho^{\lambda}:=\lambda\circ\rho\circ \lambda^{-1}$. 


\begin{defn}
The $\lambda$-scaled KP hierarchy is the hierarchy obtained by replacing $t_i$ by $\lambda_i t_i$ in the KP hierarchy (for given  $\lambda=(\lambda_i ) \in\prod_{i \in{\mathbb N}}\C^*$).
A function $\tau_1(t)\in \C[[t_1, t_2,\ldots]]$ is called $\tau$-function of the $\lambda$-scaled KP hierarchy if $\tau(\lambda^{-1}t):=\tau(\lambda_1^{-1}t_1, \lambda_2^{-1} t_2,\ldots)$ is a $\tau$-function of the KP hierarchy. For brevity, we simply say scaled KP. We do similarly for KdV, multicomponent KP. 
\end{defn}

Note that the $\lambda$-scaled KP hierarchy for $\lambda=(\mu^i)$ for $\mu\in\C^*$ coincides with the KP hierarchy. However, this does not happen in general. 

The following Lemma is the key point to go from Virasoro to KdV. 

\begin{lem}\label{lem:Diff1scaleW}
The map $\beta$ of \eqref{eq:betaDiffEnd} induces a bijection between:
\begin{itemize}
	\item the set of $ \sigma\in\Hom_{\text{Lie-alg}}({\mathcal  W}_>,\operatorname{Diff}^1(\C((z))))$ such that there exist $s\in \C$ 
	satisfying:
	$$
	 \sigma(L_i)= \frac12 z^{-2i}\big(z \partial_z + \frac{1-2i}{2}\big)+ s z^{-2i-3}  
	     $$
	\item the set of scale equivalence classes of $\rho \in \homlie({\mathcal W}_>,\End(\Vkdv))$ whose coefficients of quadratic terms in $\rho(L_{-1})$ do not vanish and such that  $\rho(L_i)$  is of type $i$ for $i\geq -1$. 
	\end{itemize}
\end{lem}

\begin{proof}
First, we prove the statement with no reference to $r(z)$ on the first item and with no mention to a linear function on the second item. Under these circumstances,  given $\sigma$ as in the statement, Lemma~\ref{lem:DinDiffistypei} shows that $(\beta\circ\sigma)(L_i) \vert_{\Vkdv}$ takes values in $\Vkdv$ and it is of type $i$ for all $i$. An explicit computation yields:
	$$
	\begin{aligned}
    (\beta\circ\sigma)(L_{-1})  \,& = \,
    s \frac{\partial}{\partial t_{1}}+
    \frac14 t_1^2 +
    \frac12 \sum_{j=1}^{\infty} j t_{j+2} \frac{\partial}{\partial t_{j}}
\\
    (\beta\circ\sigma)(L_{0}) & \,= \,
     s \frac{\partial}{\partial t_{3}} +
    \frac12 \sum_{j=1}^{\infty} j t_{j} \frac{\partial}{\partial t_{j}}
    \\
    (\beta\circ\sigma)(L_{i}) & \,= \,
    s \frac{\partial}{\partial t_{2i+3}} +
    \frac14  \sum_{j=1}^{2i-1} \frac{\partial}{\partial t_{j}} \frac{\partial}{\partial t_{2i-j}}+
    \frac12 \sum_{j=1}^{\infty} j t_{j} \frac{\partial}{\partial t_{2i-j}}
        \end{aligned}
    $$
(where $j$, as usual, is odd) and thus:
    $$
    [(\beta\circ\sigma)(L_{-1}), (\beta\circ\sigma)(L_{1}) ]
    \; - \beta({[\sigma(L_{-1}), \sigma(L_{1})]})
    \,=\,
    -\frac18
    $$
which implies that we have a map of Lie algebras defined by:
    \begin{equation}\label{eq:rhoWK}
    \rho(L_i)\,:=\, (\beta\circ\sigma)(L_i)+ \frac1{16}\delta_{i,0}
    \end{equation}

Conversely, let us start with $\rho$ as in the second set of the statement. The assumptions yield the following expression:
	$$
	\rho(L_{-1})\, = \,   b_{-1}^{0,1} \frac{\partial}{\partial t_{1}} +  a  t_1^2   + b_{-1}^{i+2,i}  t_{i+2} \frac{\partial}{\partial t_{i}}
	$$
with $ a, b_{-1}^{i+2,i}  \neq 0$. Considering the action of $\prod_{i \text{ odd}}\C^*$ by conjugation, one finds $\lambda = \{\lambda_i \in\C^* \vert i\text{ odd}\}$ and $s\in \C$ such that:
	$$
	\rho^{\lambda}(L_{-1}) \, =\, \lambda \circ \rho(L_{-1}) \circ \lambda^{-1}
	\,=\,
	s \frac{\partial}{\partial t_{1}} +  \frac14 t_1^2   + (\frac{i+2}2) t_{i+2} \frac{\partial}{\partial t_{i}}
	$$

Lemma \ref{lem:sigma} and the previous discussion show that $\rho^{\lambda}$ is the representation associated to the map  $\sigma : {\mathcal W}_>\to \operatorname{Diff}^1(\C((z)))$ defined by:
	$$
	\sigma(L_i)= \frac12 z^{-2i}(z\partial_z + \frac{1-2i}{2})+ s z^{-2i-3}
	\qquad \forall i\geq -1
	$$
\end{proof}

\begin{rem}\label{rem:VirtoKdV}
The statement can be generalized. On the one hand, we may consider the conjugation of $\sigma$ by an operator of the type  $\exp(r(z))$ while, on the other hand, we replace $\rho$ by its conjugate by $\exp(\beta(r(z)))$. For instance, for $r(z)\in \C[[z^2]]$, one has that $\beta(r(z))$ is a linear function on $t_1,t_3,\ldots$. Thus, the first representation is:
 	$$
	 \sigma(L_i)= \frac12 z^{-2i}\big(z(- r(z)+\partial_z) + \frac{1-2i}{2}\big)+ s z^{-2i-3}  
	     $$
while  $\rho$ is as in the statement up to a linear function on $t_i$'s. 
\end{rem}

\begin{rem}
It is worth noticing that the Virasoro operators studied by Witten (\cite{Witten}) correspond to the case $s=-\frac12$, $r(z)=0$. Kac-Schwarz (\cite{KacSchwarz}), using the fact the these operators come from a representation in $\operatorname{Diff}^1(\C((z)))$, proved that there is a point in the Sato Grassmannian whose $\tau$-function is a solution of these equations and, hence, is a solution of KdV hierarchy too.  A study of common solutions of Virasoro-like constraints and KdV has been carried out in \cite{Plaza-AlgSol}. 
\end{rem}

\begin{lem}\label{lem:uniquenesTauVirasoro}
Let $\rho$ be as in Lemma~\ref{lem:Diff1scaleW} and let $\tau(t)\in \Vkdv=\C[[t_1,t_3,\ldots]]$. Then, the Virasoro constraints:
	$$
	\rho(L_k)(\tau(t)) \,=\, 0\qquad k\geq -1
	$$
with the initial condition $\tau(0)=1$ admits no solution for $s=0$ and at most one solution for $s\neq 0$. 
\end{lem}

\begin{proof}
Since $\tau(0)=1$, let us consider the problem in terms of a formal function $F(t)\in \Vkdv= \C[[t_1,t_3,\ldots]]$ with $F(0)=0$ and $\tau(t)=\exp(F(t))$. The function $F(t)$ has a series expansion:
    \begin{equation}\label{eq:F=generating}
    F(t)\,=\, \sum_{\bf n} f_{\bf n} {\bf t}^{\bf n}
        \end{equation}
where ${\bf n}:=\{n_1,n_3,\ldots\}$ is a sequence of non-negative integers such that $n_i=0$ for all $i\gg 0$, $f_{\bf n}\in \C$ and ${\bf t}^{\bf n}:=\prod_{i\geq 1} t_i^{n_i}$. Further, the topology of $\Vkdv=\C[[t_1,t_3,\ldots]]$ comes from the definition $\deg(t_i)=i$. In particular, the degree of ${\bf t}^{\bf n}$ is given by $\vert {\bf n}\vert:= \sum_{i\geq 0} i n_i$.

For the sake of brevity, let us denote by  $f_{n_1 n_3\ldots n_k} = f_{\bf n}$ for ${\bf n}=\{n_1,n_3,\ldots\}$ with $n_k\neq 0$ and $n_i=0$ for all $i>k$ and we set $f_0=F(0)=0$. As a brief summary, let us write down the monomials and their coefficients up to degree $5$:
    {\small
    $$
    \begin{array}{ccccccc}
    \text{degree} & 0 & 1 & 2 & 3 & 4 & 5
    \\
    \text{monomials} & 1 & t_1 & t_1^2 & t_1^3 , t_3 & t_1^4 , t_1t_3 &
    t_1^5 , t_1^2 t_3 , t_5
    \\
    \text{coefficient} & f_0 & f_1 & f_2 & f_3, f_{01} & f_4, f_{11} &
    f_5, f_{21}, f_{001}
    \end{array}
    $$}

After rescaling $t_i$'s and conjugation by an exponential, if needed, we may assume that $\rho$ is given by \eqref{eq:rhoWK}. The hypothesis $\rho(L_k)(\tau(t)) =0$ is equivalent to the vanishing of the corresponding homogeneous parts of degree $i$ for $i=0,1,2,\ldots$. An explicit computation for low values of $k$ and $i$ yields:
	    $$
    \begin{array}{ccc}
    k &\quad  i \quad& \text{part of degree $i$ in }\rho(L_k)(\tau(t))
    \\
    -1 & 0 & s f_1 
    \\
    -1 & 1 & 2s f_2 t_1 
    \\
    -1 & 2 & 3s f_3 t_1^2 + \frac12 t_1^2 
    \\
    -1 & 3 & s\big(4 f_4 t_1^3 +f_{13} t_3\big) + t_3 f_1  
    \\
    0 & 0 & s f_{01}+\frac1{16} 
    \\
    0 & 1 &  s f_{11} t_1 + t_1 f_1  
    \\
    0 & 2 & s f_{21} t_1^2 +2 f_2 t_1^2  
    \\
    1 & 0 & s f_{001} + \frac12 f_2 +\frac14 f_1^2  
    \\
    1 & 1 & s f_{101} + \frac32 f_3 t_1 + f_{11} t_1  
    \end{array}
    $$
Thus, it is clear that if a solution $F$ does exist, then $s\neq 0$. In this case,  the vanishing of the above polynomials implies that  $f_1=0$, $f_2=0$, $f_3=-\frac1{3! s}$, $f_{01}=-\frac1{16 s}$, $f_{11}= 0$, $f_4=0$, $f_{11}= 0$, etc.~. Writing down the general expression for the homogeneous part of degree $i$ of $\rho(L_k)(\tau(t))$, one observes that it allows us to determine $f_{\bf n}$ with $\vert n\vert =i$ and $n_k\neq 0$ in terms of $f_{\bf n}$ with $\vert n\vert \leq i-2$. Thus, if a solution $F$ exists, the coefficients $f_{\bf n}$ can be recursively determined. 

\end{proof}


\begin{thm}\label{thm:solutionWscaleKdV}
Let $\rho \in \homlie({\mathcal W}_>,\End(\C[[t_1,t_3,\ldots]]))$ be such that $\rho(L_k)$  is of type $k$ for $k\geq -1$ and that all coefficients of $\rho(L_{-1})$ are non zero. 

Then, there exists a unique $\tau(t)\in\C[[t_1,t_3,\ldots]]$, with $\tau(0)=1$, such that:
	$$
	\rho(L_k)(\tau(t))\,=\,0\qquad k\geq -1
	$$
Further, the solution $\tau(t)$ is a $\tau$-function of the scaled KdV hierarchy.
\end{thm}

\begin{proof}
Lemma~\ref{lem:Diff1scaleW} implies that there is $\lambda$ and $\sigma:{\mathcal W}_>\to \operatorname{Diff}^1(\C((z)))$ such that $\rho^{\lambda}=\beta_*(\sigma)$.  Recalling  Theorem~3.12  of~\cite{Plaza-AlgSol}, one knows that there is a function $\tau_0(t)$ which satisfies that $\rho^{\lambda}(L_n)(\tau_0(t))=0$ and that it is a $\tau$-function of the KdV hierarchy. Then, $\tau(t):=\tau_0(\lambda t)$ fulfills the requirements. Since Lemma~\ref{lem:uniquenesTauVirasoro} implies the uniqueness of the solution, the conclusion follows. 

\end{proof}

\begin{rem}
Let us make two comments on the solutions. First, an instance of the notion of scaled KdV appears already in Kontsevich's Theorem when it is claimed that the exponential of the generating function in variables $T_{2i+1}:=t_i/(2i+1)!!$ is a $\tau$-function for the KdV hierarchy (\cite[Theorem 1.2]{Kon}). On the other hand, although the  dilation shift $\bar t_{i}\mapsto \bar t_{i}-\delta_{i,0}$ transforms the operators $\rho(L_k)$, it should be noted that it does not induce an automorphism of the algebra $\C[[\bar t_0,\bar t_1,\ldots]]$.
\end{rem}

%

\subsection{On the solutions for the $n$-dimensional situation}

Let us now focus in the $n$-dimensional situation. That is, we aim at studying the interplay between Virasoro representations and multicomponent KP hierarchy. Special attention will be paid at their common solutions. 

Recall that $\Vkdv(A)$ is the subalgebra of $\C[[t_1,t_3,\ldots]]\widehat\otimes_{\C} S^\bullet A$ generated by $t_i\otimes a$. Then, $S\in\Gl(A)$ acts on it by the automorphism of algebras $t_i\otimes a\mapsto t_i\otimes S(a)$. 

\begin{thm}\label{thm:solutionproductKdV}
Let $\rho:{\mathcal W}_>\to \uh[(A)]$ be as in \S\ref{subsec:baby}.

There exist $S\in \Gl(A)$ and functions $\tau_\alpha(t_{1,\alpha},t_{3,\alpha},\ldots)\in \C[[t_{1,\alpha},t_{3,\alpha},\ldots]]$ such that:
	\begin{equation}\label{eq:scaleequivmultiKP}
	\hat\rho(L_k)\big(S( \prod_\alpha \tau_\alpha( t_\alpha))\big)\,=\, 0
	\end{equation}

Further, $\tau_\alpha(t_{1,\alpha},t_{3,\alpha},\ldots)$ are $\tau$-functions of the scaled KdV hierarchy.
 \end{thm}

\begin{proof}
Theorem~\ref{thm:Virasorodecomp} shows that there is $S\in\Gl(A)$ such that $\rho^S$ decomposes as the tensor product of $n$ $1$-dimensional Lie algebra representations of ${\mathcal W}_>$. More precisely,  if $\{a_\alpha\}$ is the chosen basis for $A$, then $\{S(a_\alpha)\}$ is a orthogonal basis for $\eta$. Consequently, there are $\rho_\alpha: {\mathcal W}_>\to \uh[(<S(a_\alpha)>)]$ such that \eqref{eq:rho=sumrhoalpha} holds. 

Now, apply the results of \S\ref{subsec:solutions1dim} on the $1$-dimensional case. Indeed, since $\eta$ is non-degenerated and $\{S(a_\alpha)\}$ is a orthogonal basis,  from Theorem~\ref{thm:solutionWscaleKdV} one obtains functions $\tau_\alpha(t_\alpha)$, such that $\tau_\alpha(0)=1$,  $\rho_\alpha(L_k)(\tau_\alpha)=0$ for all $\alpha,k$ and they are  $\tau$-functions for the scaled KdV hierarchy. 

Observe that \eqref{eq:scaleequivmultiKP} holds if and only if $\hat\rho^S(L_k)( \prod_\alpha \tau_\alpha( t_\alpha))$ vanishes. Applying the converse of Theorem~\ref{thm:solutionrhoi} one concludes. 
\end{proof}


\begin{rem}
The previous Theorem means that, assuming the uniqueness of the solution (\cite[Theorem 3.10.20]{DubrovinZ}), the solution of the Virasoro constraints has to be of the above form; that is, an operator acting on a product of Witten-Kontsevich $\tau$-functions. Thus, it agrees with the results of Givental (\cite{Givental}) for the total descendent potential. It would be interesting to relate both expressions explicitly (see also \cite{FvLS, GiventalAn,Lee}). Alternatively, one could combine Teleman's classification of semi simple cohomological field theories (\cite{Teleman}) with Givental's results to deduce that this is the right expression for the solution. Nevertheless our result can be applied on other frameworks, as it will mentioned in \S\ref{subsec:final}).
\end{rem}

%
%
%

\begin{cor}
Let $\rho$ be as in the Theorem~\ref{thm:solutionproductKdV}. 

If $S, \tau_\alpha$ satisfy \eqref{eq:scaleequivmultiKP}, then $\rho^S=\rho_1+\ldots+\rho_n$ and $\hat\rho_\alpha(L_k)(\tau_\alpha)=0$. 

The matrix $S$ is unique up to an ortogonal matrix. 
\end{cor}

\begin{proof}
If $S$ and $\tau_\alpha$ are such that \eqref{eq:scaleequivmultiKP} vanishes, then the following expression also vanishes:
	$$
	0\,=\, 
	{\exp(-\sum_\alpha \tilde\tau_\alpha(t_\alpha))} S^{-1} \hat\rho(L_k)\big(S( \prod_\alpha \tau_\alpha( t_\alpha))\big) 
	\,=\, 
	\frac{ \hat \rho^S(L_k)(\exp(\sum_\alpha \tilde\tau_\alpha(t_\alpha)))} {\exp(\sum_\alpha \tilde\tau_\alpha(t_\alpha))}
	$$
Recall that an operator $\rho^S(L_k)$ of type \eqref{eq:type} is the same as $\rho(L_k)$  where the matrix $a$  has been replaced by $(S^{-1})^T a S^{-1}$ (and, accordingly, $b_k$, $c_k$, etc.). Expanding the case $k=-1$ of the last identity, one obtains that $(S^{-1})^T \eta S^{-1}$ is diagonal. Then, Theorem~\ref{thm:rhodecomp}, implies that $\rho^S$ decomposes as a sum $\rho_1+\ldots+\rho_n$ and Theorem~\ref{thm:solutionrhoi} implies that $\hat\rho_\alpha(L_k)(\tau_\alpha(t_\alpha))=0$. 

It is straightforward that $S$ is unique up to an ortogonal matrix.
\end{proof}

%
%

\begin{cor}
Let $\rho$ be as in the Theorem~\ref{thm:solutionproductKdV}.  If either $S$ is diagonal or $\tau_\alpha$ are $\tau$-functions of the same scaled hierarchy, then the solution is a $\tau$-function for the scaled multicomponent KP hierarchy. 
\end{cor}

\begin{proof}
In particular, the product $S( \prod_\alpha \tau_\alpha( t_\alpha))$ is uniquely determined by $\rho$. Each function $\tau_\alpha( t_\alpha)$ satisfies the scaled KdV and, thus, there are $\lambda_\alpha:=(\lambda_{i,\alpha})\in\prod_{i\text{ odd}}\C^*$ such that $\tau_\alpha( \lambda_\alpha^{-1} t_\alpha)$ defines a point $U_\alpha\in \Gr(\C((z)))$. If $\rho$ is expressed w.r.t. a basis $\{a_1,\ldots , a_n\}$, then $S$ determines a second basis $\{S(a_1),\ldots, S(a_n)\}$ or, equivalently, an isomorphism  $\C\oplus\ldots \oplus\C\simeq A$. This isomorphism induces:
	{\small $$
	\Gr(\C((z)))\times \ldots\times \Gr(\C((z))) \hookrightarrow 
	\Gr(\C((z))\oplus\ldots\oplus\C((z))) 
	\simeq \Gr(A\otimes \C((z)))
	$$}
Since $\tau$-function of the image of $(U_1,\ldots, U_n)$, which is $U_1\oplus\ldots\oplus U_n$, is given by $\prod_\alpha \tau_\alpha(\lambda_\alpha^{-1} t_\alpha)$ it follows that $S \prod_\alpha \tau_\alpha(t_\alpha)$ is a $\tau$-function of the scaled multicomponent KP in the two cases of the statement. 
\end{proof}


\begin{rem}
Recalling Remark~\ref{rem:VirtoKdV}, we observe that Theorems~\ref{thm:solutionWscaleKdV} and~\ref{thm:solutionproductKdV} could be weaken and stated for representations satisfying the hypothesis up to a linear function on $t_i$'s.
\end{rem}

\subsection{Final Comments}\label{subsec:final}

Let us finish with some brief comments. From a general perspective, we hope that our methods shed some light on the explicit expressions of the Virasoro operators and of the relevant integrable hierarchies that appear in   the Virasoro conjecture.  Furthermore, they can also be applied to many instances of representations of ${\mathcal W}_>$ such as recursion relations, Hurwitz numbers, knot theory, etc.~. 

As an illustration, let us point out the results of \cite{AmbjornChekhov,KazarianZograf} on Hurwitz numbers. In both cases, the authors study the generating functions of the number of coverings of ${\mathbb P}_1\setminus\{0,1,\infty\}$ with some properties. It is shown that these functions satisfy Virasoro constraints, KP hierarchy and topological recursion (of the Eynard-Orantin type \cite{EO}). It is remarkable that the Virasoro constraints are explicitly expressed as differential operators of the form considered in \S\ref{sec:ReprW>} for the case $A=\C$. Thus, the results of \S\ref{subsec:solutions1dim} can be directly applied to conclude that Virasoro constraints imply the scaled KP hierarchy. 


Our results could also be of interest within the context of Eynard-Orantin topological recursion (\cite{EO}). Indeed,  we learn from \cite{MulaseSafnuk} that Mirzakhani's recursion formula for the Weil-Petersson volumes (\cite{Mir}) is indeed a Virasoro constraint imposed on a generating function of these volumes and that this function satisfies the KdV hierarchy. It is worth pointing out some recent results on the relation of topological recursion and Virasoro constraints (\cite{Eynard, Milanov}). On the one hand, it has been shown in \cite{Eynard} that these Virasoro constraints are actually equivalent to Eynard-Orantin topological recursion for some spectral curve. On the other hand, one knows from \cite{Milanov}  that the correlation functions of a semisimple cohomological field theory satisfy the Eynard-Orantin topological recursion and that these recursion formulas are equivalent to $n$ copies of the Virasoro constraints for the ancestor potential. Therefore, two problems can be faced with our techniques. First, we think that Theorem~\ref{thm:solutionproductKdV} should imply  some bilinear relations of Hirota type for the solution of the Eynard-Orantin topological recursion. Second, due to the uniqueness of the solution and the fact that the solution satisfies the KP hierarchy, there must be a relation of the Eynard-Orantin spectral curve and the Krichever construction.  

Similarly, it would be interesting to interpret the recent papers \cite{Da,DOSS} from our perspective.


\begin{thebibliography}{\hphantom{10pt}}

\bibitem{AMP} \'Alvarez V\'azquez, A.; Mu\~noz Porras, J. M.; Plaza Mart\'{\i}n, F. J., ``The algebraic formalism of soliton equations over arbitrary base fields'' in \emph{Workshop on Abelian Varieties and Theta Functions}, 3--40, 
Aportaciones Mat. Investig., 13, Soc. Mat. Mexicana, M\'exico, 1998. 

\bibitem{AmbjornChekhov} Ambj\o rn, J.; Chekhov, L., ``The matrix model for dessins d'enfants'',  Ann. Inst. Henri Poincaré D 1 (2014), no. 3, 337--361

\bibitem{Da} Dartois, S., ``A Givental-like Formula and Bilinear Identities for Tensor Models'', J. High Energy Phys. 2015, no. 8, 129 

\bibitem{Dijkgraaf} Dijkgraaf, R., ``Intersection Theory, Integrable Hierarchies and Topological Field Theory'', in ``New Symmetry Principles in Quantum Field Theory,'' Ed. Mack (Plenum, 1993) 95-158, \texttt{arXiv:hep-th/9201003}

\bibitem{DVV} Dijkgraaf, R.; Verlinde, H.; Verlinde, E., ``Loop equations and Virasoro constraints in non-perturbative 2D quantum gravity'', Nuclear Physics B, volume 348 (1991), issue 3, 435-456 



\bibitem{DubrovinZ} Dubrovin, B.; Zhang, Y.,  ``Frobenius manifolds and Virasoro constraints'', Selecta Math. (N.S.) 5 (1999), no. 4, 423--  466.

\bibitem{DubrovinZ2} Dubrovin, B.; Zhang, Y.,  ``Normal forms of hierarchies of integrable PDEs, Frobenius manifolds and Gromov - Witten invariants'',  \texttt{arXiv:0108160}

\bibitem{DOSS} Dunin-Barkowski, P.; Orantin, N.; Shadrin, S.; Spitz, L., ``Identification of the Givental formula with the spectral curve topological recursion procedure'', Comm. Math. Phys. 328 (2014), no. 2, 669--700.

\bibitem{EHX} Eguchi, T.; Hori, K.; Xiong, C., ``Quantum cohomology and Virasoro algebra'', Phys. Lett. B 402 (1997), no. 1-2, 71--  80.

\bibitem{Eynard} Eynard, B., ``Recursion between Mumford volumes of moduli spaces'', Ann. Henri Poincaré 12 (2011), no. 8, 1431--1447

\bibitem{EO} Eynard, B.; Orantin, N., ``Invariants of algebraic curves and topological expansion'', Commun. Number Theory Phys. 1 (2007), no. 2, 347--452

\bibitem{FvLS} Feigin, E.; van de Leur, J.; Shadrin, S., ``Givental symmetries of Frobenius manifolds and multi-component KP tau-functions'', Adv. Math. 224 (2010), no. 3, 1031--1056

\bibitem{Getzler} Getzler, E., ``The Virasoro conjecture for Gromov-Witten invariants'', in \emph{Algebraic geometry: Hirzebruch 70} (Warsaw, 1998), 147--  176, Contemp. Math., 241, Amer. Math. Soc., Providence, RI, 1999

\bibitem{Givental}  Givental, A., ``Gromov-Witten invariants and quantization of quadratic Hamiltonians'', Mosc. Math. J., 2001, Volume 1, Number 4, pp.  551--  568

\bibitem{GiventalAn}  Givental, A., ``$A_{n-1}$ singularities and $n$KdV hierarchies'', 
Mosc. Math. J. 3 (2003), no. 2, 475--505


\bibitem{HMT} Hertling, C.; Manin, Y.; Teleman, C., ``An update on semisimple quantum cohomology and F-manifolds'', Proc. Steklov Inst. Math. 264 (2009), no. 1, 62--69

\bibitem{JarvisKimura}  Jarvis, T. J.; Kimura, T.,  ``Orbifold quantum cohomology of the classifying space of a finite group'', in \emph{Orbifolds in mathematics and physics} (Madison, WI, 2001), 123--  134, Contemp. Math., 310, AMS Providence, RI, 2002.



\bibitem{KacSchwarz} Kac, V.; Schwarz, A., ``Geometric interpretation of the partition function of 2D gravity'', Physics Letters B, vol. 257 (1991), pp. 329--334

\bibitem{KSU} Katsura, T.; Shimizu, Y.; Ueno, K., ``Complex cobordism ring and conformal field theory over ${\mathbb Z}$'', Math. Ann. 291 (1991), pp. 551--571

\bibitem{Kaz} Kazarian, M., ``KP hierarchy for Hodge integrals', Advances in Mathematics 221 (2009), pp. 1--21

\bibitem{KazarianZograf} Kazarian, M.; Zograf, P., ``Virasoro Constraints and Topological Recursion for Grothendieck’s Dessin Counting'', Lett. Math. Phys. 105 (2015), no. 8, 1057--1084

\bibitem{Kon} Kontsevich, M., ``Intersection theory on the moduli space of curves and the matrix Airy function'',  Comm. Math. Phys.  147  (1992),  no. 1, 1--23.

\bibitem{Lee} Lee, Y., ``Invariance of tautological equations. II. Gromov-Witten theory'', J. Amer. Math. Soc. 22 (2009), no. 2, 331--352. 

\bibitem{LW} Libgober, A.; Wood, J., ``Uniqueness of the complex structure on Kähler manifolds of certain homotopy types'', J.  Diff. Geom. 32 no.1 (1990), 139--154

\bibitem{LiuUnique} Liu, S.; Yang, D.; Zhang, Y., ``Uniqueness theorem of W-constraints for simple singularities'', Lett. Math. Phys. 103 (2013), no. 12, 1329--1345

\bibitem{LiuTian} Liu, X.; Tian, G., ``Virasoro constraints for quantum cohomology'', J. Differential Geom. 50 (1998), no. 3, 537--590

\bibitem{Milanov} Milanov, T., ``The Eynard-Orantin recursion for the total ancestor potential'', Duke Math. J. 163 (2014), no. 9, 1795--1824

\bibitem{Mir} Mirzakhani, M., ``Simple geodesics and Weil-Petersson volumes of moduli spaces of bordered Riemann surfaces'', Invent. Math. 167 (2007), no. 1, 179--222

\bibitem{MulaseSafnuk} Mulase, M.; Safnuk, B., ``Mirzakhani's recursion relations, Virasoro constraints and the KdV hierarchy'',  Indian J. Math. 50 (2008), no. 1, 189--218

\bibitem{OK}  Okounkov, A.; Pandharipande, R., ``Virasoro constraints for target curves'', Invent. Math. 163 (2006), no. 1, 47--108

\bibitem{OK-Toda}  Okounkov, A.; Pandharipande, R., ``The equivariant Gromov-Witten theory of ${\mathbb P}^1$'', Ann. of Math. (2) 163 (2006), no. 2, 561--605

\bibitem{Plaza-AlgSol} Plaza Mart\'{\i}n, F., ``Algebro-Geometric Solutions of the Generalized Virasoro Constraints'', SIGMA 11 (2015), 052, 34 pages, \texttt{arXiv:1110.0729}

\bibitem{Plaza-Opers} Plaza Mart\'{\i}n, F., ``Representations of the Witt algebra and Gl(n)-opers'', Lett. Math. Phys. 103 (2013), no. 10, 1079--1101

\bibitem{PT} Plaza Mart\'{\i}n, F.; Tejero-Prieto, C., ``Extending representations of $\sl(2)$ to Witt and Virasoro algebras'', \texttt{arXiv:1411.5598v1}

\bibitem{Sato} Sato, M.; Sato, Y., ``Soliton equations as dynamical systems on infinite-dimensional Grassmann manifold", in \emph{ Nonlinear partial differential equations in applied science}, 259--271, 
North-Holland Math. Stud., 81 (1983). 

\bibitem{SW} Segal, G.; Wilson, G., ``Loop groups and equations of KdV type'', 
Inst. Hautes Études Sci. Publ. Math. No. 61 (1985), 5--65. 


\bibitem{Teleman} Teleman, C., ``The structure of 2D semi-simple field theories'', Invent. Math. 188 (2012), no. 3, 525--588


\bibitem{Witten} Witten, E., ``Two dimensional gravity and intersection theory on moduli space'', Surveys in Diff. Geom. 1, 243-310 (1991)

\end{thebibliography}
\end{document}